\documentclass[MSNbibl,number,citesort,seceqn,dvips]{arxbj}
\usepackage{upgreek}

\aid{0}
\volume{19}
\issue{2}
\pubyear{2013}
\firstpage{548}
\lastpage{598}
\doi{10.3150/11-BEJ403}

\makeatletter
\renewcommand{\subset}{\subseteq}

\newtheorem{theorem}{Theorem}[section]
\newtheorem{lemma}[theorem]{Lemma}
\newtheorem{corollary}[theorem]{Corollary}
\newtheorem{proposition}[theorem]{Proposition}
\newproclaim{definition}[theorem]{Definition}
\newremark{remark}[theorem]{Remark}
\newremark{example}[theorem]{Example}
\newcommand{\mb}[1]{\mathbf{#1}}
\newcommand{\mbb}[1]{\mathbb{#1}}
\newcommand{\wt}[1]{\widetilde{#1}}
\newcommand{\wh}[1]{\widehat{#1}}
\newcommand{\ppr}[5]
{\frac{{#1}_{(#2)}}{\prod_{i=1}^{#3}{({#4}_{i})}_{({#5}_{i})}}}
\newcommand{\copr}[4]{\prod_{i=1}^{#1}({#2}_{i}{#3}_{i})^{{#4}_{i}}}%
\newcommand{\bmom} [3] {[(1-#1)(1-#2)]^{#3}} 
\def\ifact#1#2{{#1}_{(#2)}}
\makeatother

\begin{document}
\begin{frontmatter}

\title{Orthogonal polynomial kernels and canonical correlations for Dirichlet measures}
\runtitle{Dirichlet polynomial kernels}

\begin{aug}
\author[1]{\fnms{Robert C.} \snm{Griffiths}\thanksref{1}\ead[label=e1]{griff@stats.ox.ac.uk}}
\and
\author[2]{\fnms{Dario} \snm{Span\`{o}}\corref{}\thanksref{2}\ead[label=e2]{d.spano@warwick.ac.uk}}
\runauthor{R.C. Griffiths and D. Span\`{o}} 
\address[1]{Department of Statistics, University of Oxford,
1 South Parks Road Oxford OX1 3TG, UK.\\
\printead{e1}}
\address[2]{Department of Statistics, University of Warwick,
Coventry CV4 7AL, UK.\\ \printead{e2}}
\end{aug}

\received{\smonth{1} \syear{2011}}
\revised{\smonth{9} \syear{2011}}

%
\begin{abstract}
We consider a multivariate version of the so-called Lancaster problem
of characterizing canonical correlation coefficients of symmetric
bivariate distributions with identical marginals and orthogonal
polynomial expansions. The marginal distributions examined in this
paper are the Dirichlet and the Dirichlet multinomial distribution,
respectively, on the continuous and the $N$-discrete $d$-dimensional
simplex. Their infinite-dimensional limit distributions, respectively,
the Poisson--Dirichlet distribution and Ewens's sampling formula, are
considered as well. We study, in particular, the possibility of mapping
canonical correlations on the $d$-dimensional continuous simplex (i) to
canonical correlation sequences on the $d+1$-dimensional simplex and/or
(ii) to canonical correlations on the discrete simplex, and vice versa.
Driven by this motivation, the first half of the paper is devoted to
providing a full characterization and probabilistic interpretation of $
n$-orthogonal polynomial kernels (i.e., sums of products of orthogonal
polynomials of the same degree $n$) with respect to the mentioned
marginal distributions. We establish several identities and some
integral representations which are multivariate extensions of important
results known for the case $d=2$ since the 1970s. These results, along
with a common interpretation of the mentioned kernels in terms of
dependent P\'{o}lya urns, are shown to be key features leading to
several non-trivial solutions to Lancaster's problem, many of which can
be extended naturally to the limit as $d\rightarrow\infty.$
\end{abstract}

%

\end{frontmatter}

\section{Introduction}\label{sec:intro1}

Let $\pi$ be a probability measure on some Borel space
$(E,\mathcal{E})$ with $E\subseteq\mbb{R}.$ Consider an exchangeable
pair $(X,Y)$ of random variables with given marginal law $\pi.$
Modeling tractable joint distributions for $(X,Y)$, with $\pi$ as given
marginals, is a classical problem in mathematical statistics. One
possible approach, introduced by Henry Oliver Lancaster~\cite{L58} is
in terms of so-called \textit{canonical correlations}. Let $\{P_n\}
_{n=0}^{\infty}$ be a family of orthogonal polynomials with weight
measure $\pi$, that is, such that
\[
\mbb{E}_{\pi}(P_n(X)P_m(X))=\frac{1}{c_m}\delta_{nm},\qquad n,m\in\mbb{Z}_+
\]
for a sequence of positive constants $\{c_m\}.$ Here $\delta_{mn}=1$ if
$n=m$ and $0$ otherwise, and
$\mbb{E}_{\pi}$ denotes the expectation taken with respect to
$\pi.$

A sequence $\rho=\{\rho_n\}$ is the sequence of canonical correlation
coefficients for the pair $(X,Y)$, if it is possible to write the joint
law of $(X,Y)$ as
%
%
\begin{equation}
g_{\rho}(\mathrm{d}x,\mathrm{d}y)=\pi(\mathrm{d}x)\pi(\mathrm{d}y)\Biggl\{\sum_{n=0}^{\infty}\rho
_nc_nP_n(x)P_n(y)\Biggr\},\label{grho}
\end{equation}
where $\rho_0=1.$ Suppose that the system $\{P_n\}$ is \textit{complete}
with respect to $L_2(\pi)$; that is, every function $f$ with finite
$\pi
$-variance admits a representation
%
%
\begin{equation}
f(x)=\sum_{n=0}^{\infty}\widehat{f}(n)c_nP_n(x),
\label{riesz}
\end{equation}
where
%
%
\begin{equation}
\widehat{f}(n)=\mbb{E}_{\pi}[f(X)P_n(X)],\qquad n=0,1,2,\ldots.
\end{equation}
Define the conditional expectation operator by
\[
T_{\rho} f(x):=\mathbb{E}\bigl(f(Y)| X=x\bigr).
\]

If $(X,Y)$ have canonical correlations $\{\rho_n\}$, then, for every
$f$ with finite variance,
\[
T_{\rho}f(x)=\sum_{n=0}^{\infty}\rho_n\wh{f}(n)c_nP_n(x).
\]
In particular,
\[
T_{\rho} P_n=\rho_nP_n,\qquad n=0,1,\ldots;
\]
that is, the polynomials $\{P_n\}$ are the eigenfunctions, and $\rho$
is the sequence of eigenvalues of $T_{\rho}.$
Lancaster's problem is therefore a spectral problem, whereby
conditional expectation operators with given eigenfunctions are
uniquely characterized by their eigenvalues.
Because $T_{\rho}$ maps positive functions to positive functions, the
problem of identifying canonical correlation sequences~$\rho$ is
strictly related to the problem of characterizing so-called \textit{positive-definite sequences}.

In this paper we consider a multivariate version of Lancaster's
problem, when $\pi$ is taken to be either the Dirichlet or the
Dirichlet multinomial distribution (notation: $D_{\alpha}$ and
$\operatorname{DM}_{\alpha,N},$ with $\alpha\in\mbb{R}^{d}_{+}$ and $N\in\mbb{Z}_+$)
on the $(d-1)$-dimensional continuous and $N$-discrete simplex, respectively.
The eigenfunctions will be the multivariate Jacobi or Hahn
polynomials, respectively.
One difficulty arising when $d>2$ is that the orthogonal polynomials
$P_{\mb{n}}=P_{n_1n_2\cdots n_d}$ are multi-indexed. The degree of
every polynomial $P_{\mb{n}}$ is
$|\mb{n}|:=n_1+\cdots+n_d$ (throughout the paper, for every vector
$\mb
{x}=(x_1,\ldots,x_d)\in\mbb{R}^d$, we will denote its length by
$|\mb
{x}|$). There are
\[
\pmatrix{ n + d -1\vspace*{2pt}\cr d-1}
\]
polynomials with degree $ n,$ so, when $d>2$, there is no unique way to
introduce a total order in the space of all
polynomials. \textit{Orthogonal polynomial kernels} are instead uniquely
defined and totally ordered.

By orthogonal polynomial kernels of degree $ n$, with respect to $\pi$,
we mean functions of the form
%
%
\begin{equation}
P_{ n}(\mb{x},\mb{y})=\sum_{\mb{m}\in\mbb{Z}_+^d:|\mb{m}|=
n}c_{\mb
{m}}P_{\mb{m}}(\mb{x})P_{\mb{m}}(\mb{y}),\qquad
n=0,1,2,\ldots,
\label{gker}
\end{equation}
where $\{P_{\mb{n}}\dvt \mb{n}\in\mbb{Z}_+^d\}$ is a system of orthogonal
polynomials with weight measure $\pi.$

It is easy to check that
\[
\mbb{E}_{\pi}[P_{ n}(\mb{x},\mb{Y})P_{ m}(\mb{z},\mb{Y})
]=P_{ n}(\mb{x},\mb{z})\delta_{ m n}.
\]

A representation equivalent to (\ref{riesz}) in term of polynomial
kernels is
%
%
\begin{equation}
f(\mb{x})=\sum_{ n=0}^{\infty}\mbb{E}_{\pi}(f(\mb{Y})P_{ n}(\mb
{x},\mb{Y})).
\label{rieszk}
\end{equation}
If $f$ is a polynomial of order $ m$, the series terminates at
$ m$. Consequently, for general $d\geq2$, the individual
orthogonal polynomials $P_{\mb{n}}(\mb{x})$ are uniquely determined
by their
leading coefficients of degree $ n$ and $P_{ n}(\mb{x},\mb{y})$. If a
leading term is
\[
\sum_{\{\mb{k}:|\mb{k}|= n\}}b_{{\mb{n}}{\mb{k}}}\prod_1^\mathrm{d}x_i^{k_i},
\]
then
%
%
\begin{equation}
P_{\mb{n}}(\mb{x}) =
\sum_{\{\mb{k}:|\mb{k}|= n\}}b_{{\mb{n}}{\mb{k}}}\mbb{E}\Biggl[\prod
_1^dY_i^{k_i}P_{ n}(\mb{x},\mb{Y})
\Biggr], \label{leadingQ}
\end{equation}
where $Y$ has distribution $\pi$.

$P_{ n}(\mb{x},\mb{y})$ also has an expansion in terms of any
complete sets of biorthogonal polynomials of
degree $ n$. That is, if $\{P_{\mb{n}}^\diamond(\mb{x})\}$ and
$\{P_{\mb{n}}^\circ(\mb{x})\}$ are polynomials orthogonal to
polynomials of degree less
that $ n$ and
\[
\mbb{E}[P^\diamond_{\mb{n}}(X)P^\circ_{\mb{n}^\prime} (X)] =
\delta_{\mb{n}\mb{n}^\prime},
\]
then
%
%
\begin{equation}
P_{ n}(\mb{x},\mb{y}) = \sum_{\{\mb{n}:|\mb{n}|=n\}}P^\diamond
_{\mb
{n}}(\mb{x})P^\circ_{\mb{n}}(\mb{y}).
\end{equation}
%
Similar expressions to (\ref{leadingQ}) hold for $P_{\mb{n}}^\diamond
(\mb{x})$ and $P_{\mb{n}}^\circ(\mb{x})$, using their respective
leading coefficients. This can be
shown by using their expansions in an orthonormal polynomial set and
applying~(\ref{leadingQ}).

The polynomial kernels with respect to $D_\alpha$ and $\operatorname{DM}_{\alpha,N}$
will be denoted by $Q_{ n}^\alpha(\mb{x},\mb{y})$ and $H_{ n}^\alpha
(\mb
{r},\mb{s}),$ and called Jacobi and Hahn kernels, respectively.

This paper is divided in two parts. The goal of the first
part is to describe Jacobi and Hahn kernels under a unified view: we
will first provide a probabilistic description of their structure and
mutual relationship, then we will investigate their symmetrized and
infinite-dimensional versions.

We will show that all the kernels
under study can be constructed via systems of bivariate P\'{o}lya urns
with random samples in common. This remarkable property assimilates the
Dirichlet ``world'' to other distributions, within the
so-called Meixner class, whose orthogonal polynomial kernels admit a
representation in terms of bivariate sums with random elements in
common, a~fact known since the 1960s (see~\cite{E64,EL67}. See also
\cite{DKSC08} for a modern Bayesian approach).

In the second part of the paper we will turn our
attention to the problem of identifying
canonical correlation sequences with respect to $D_{\alpha}$
and $\operatorname{DM}_{\alpha,N}$.
We will restrict our focus on sequences~$\rho$ such that, for every
$\mb{n}\in\mbb{Z}_+^d,$ $\rho_{\mb{n}}$ depends on $\mb{n}$ only
through its total length $|\mb{n}|=\sum_{i=1}^dn_i$:
\[
\rho_{\mb{n}}=\rho_{ n}    \qquad   \forall\mb{n}\in\mbb
{Z}_+^d\dvt |\mb{n}|=n.
\]
For these sequences, Jacobi or Hahn polynomial kernels will be used to
find out conditions for a sequence $\{\rho_{ n}\}$ to satisfy the inequality
%
%
\begin{equation}\sum_{ n=0}^{\infty}\rho_{ n}P_{ n}(\mb{u},\mb
{v})\label
{cop2}\geq0.
\end{equation}
Since $T_{\rho}$ is required to map constant functions to constant
functions, a straightforward necessary condition is always that
\[
\rho_{0}=1.
\]
For every $d=2,3,\ldots$ and every $\alpha\in\mbb{R}_+^d$, we will call
any solution to (\ref{cop2}) an \textit{$\alpha$-Jacobi positive-definite
sequence \textup{(}$\alpha$-JPDS\textup{)}}, if $\pi=D_\alpha$, and an $(\alpha
,N)$-\textit{Hahn positive-definite sequence} ($(\alpha,N)$-\textit{HPDS}), if $\pi
=\operatorname{DM}_{\alpha,N}$.

We are interested, in particular, in studying if and when one or both
the following statements are true.
\begin{enumerate}[(P1)]
\item[(P1)] For every $d$ and $\alpha\in\mbb{R}_+^d$, $\rho$ is
$\alpha
$-JPDS $\Leftrightarrow \rho$ is $\wt{\alpha}$-JPDS for every
$\wt
{\alpha}\in\mbb{R}^{d+1}_+\dvt |\wt{\alpha}|=|\alpha|;$

\item[(P2)] For every $d$ and $\alpha\in\mbb{R}_+^d$ $\rho$ is
$\alpha
$-JPDS $\Leftrightarrow \rho$ is $(\alpha,N)$-HPDS for some $N$.
\end{enumerate}
Regarding (P1), it will be clear in Section~\ref{sec:pdandk} that the
sufficiency part ($\Leftarrow$) always holds. To find conditions for
the necessity part ($\Rightarrow$) of (P1), we will use two alternative
approaches. The first one is based on a multivariate extension of a
powerful product formula for the Jacobi polynomials, due to Koornwinder
and finalized by Gasper in the early 1970s: for $\alpha,\beta$ in a
``certain region'' (see Theorem~\ref{th:gasper} further on), the
integral representation
\[\frac{P^{\alpha,\beta}_n(x)}{P^{\alpha,\beta
}_n(1)}\frac{P^{\alpha,\beta}_n(y)}{P^{\alpha,\beta}_n(1)}=
\int_0^1
\frac{P^{\alpha,\beta}_{n}(z)}{P^{\alpha,\beta
}_n(1)}m_{x,y}(\mathrm{d}z),\qquad x,y\in(0,1),n\in\mbb{N},\label{pf}
\]
holds for a probability measure $m_{x,y}$ on [0, 1]. Our extension for
multivariate polynomial kernels, of non-easy derivation, is found in
Proposition~\ref{PROP:1} to be
%
%
\begin{equation}
{Q_{ n}^{\alpha}(\mb{x},\mb{y})}=\mbb{E}\bigl[{Q^{\alpha_d, |\alpha
|-\alpha_d}_{ n}(Z_d,1)} |  \mb{x},\mb{y}\bigr],\qquad  |n|=0,1,\ldots
\label{protoprop}
\end{equation}
for every $d\in\mbb{N},$ every $\alpha$ in a ``certain
region'' of $\mbb{R}_+^d,$ and for a particular $[0,1]$-valued random
variable $Z_d$. Here, for every $j=1,\ldots,d,$ $\mb{e}_j=(0,0,\ldots
,1,0,\ldots,0)\in\mbb{R}^d$ is the vector with all zero components,
except for the $j$th coordinate which is equal to 1.
Integral representations such as (\ref{protoprop}) are useful in that
they map immediately univariate positive functions to the type of
bivariate distribution we are looking for,
\[
f(\mb{x})\geq0 \quad\Longrightarrow\quad\sum_{ n}\wh{f}(n)Q_{ n}(\mb{x},\mb{y})=
\mbb{E}[f(Z_d) | \mb{x},\mb{y}]\geq0.
\]
In fact, whenever (\ref{protoprop}) holds true, we will be able to
conclude that (P1) is true.

Identity (\ref{protoprop}), however, holds only with particular choices
of the parameter $\alpha$. At best, one needs one of the $\alpha_j$s to
be greater than 2. This makes it hard to use (P1) to build, in the
limit as $d\to\infty,$ canonical correlations with respect to
Poisson--Dirichlet marginals on the infinite simplex. The latter would
be a desirable aspect for modeling dependent measures on the infinite
symmetric group or for applications, for example, in nonparametric
Bayesian statistics.

On the other hand, there are several examples in the literature of
positive-definite sequences satisfying (P1) for every choice of $\alpha
,$ even in the limit case of $|\alpha|=0$. Two notable and well-known
instances are
\begin{enumerate}[(ii)]
\item[(i)]
%
%
\begin{equation}\rho_{ n}(t)=\mathrm{e}^{-({1}/{2}) n( n+|\alpha
|-1)t},\qquad  n=0,1,\ldots,\label{geneig}
\end{equation}
arising as the eigenvalues of the transition semigroup of the so-called
\textit{$d$-type, neutral Wright--Fisher diffusion process in population
genetics; see, for example},~\mbox{\cite{W84,G79a,GS10}}. The generator of
the diffusion process $\{\mb{X}(t), t \geq0\}$ describing the relative
frequencies of genes with type space $\{1,\ldots,d\}$ is
\[
\mathcal{L} = \frac{1}{2}\sum_{i=1}^d\sum_{j=1}^\mathrm{d}x_i(\delta_{ij}-x_j)
\frac{\partial^2}{\partial x_i\,\partial x_j}
+ \frac{1}{2} \sum_{i=1}^d(\alpha_i - |\alpha|x_i)\frac{\partial
}{\partial x_i}.
\]
In this model, mutation is parent-independent from type $i$ to $j$ at
rate $\alpha_j/2$, $j \in\{1,\ldots,d\}$. Assuming that $\alpha>0$,
the stationary distribution of the process is ${ D}_\alpha$, and the
transition density has an expansion
\[
f(\mb{x},\mb{y};t) = { D}_\alpha(y)\Biggl\{
1 + \sum_{n=1}^\infty\rho_n(t)Q_n^\alpha(\mb{x},\mb{y})\Biggr\}.
\]
The limit model as $d\to\infty$ with $\alpha= |\alpha|/d$ is the
infinitely-many-alleles-model, where mutation is always to a novel
type. The stationary distribution is $\operatorname{Poisson\mbox{--}Dirichlet} (\alpha)$.

The same sequence (\ref{geneig}) is also a HPDS playing a role in
population genetics~\cite{KMG75}: it is the eigenvalue sequence of the
so-called Moran model with type space $\{1,\ldots,d\}$. In a population
of $N$ individuals, $\{\mb{Z}(t),t \geq0\}$ denotes the number of
individuals of each type at $t$, $|\mb{Z}(t)|=N$. In reproduction
events, an individual is chosen at random to reproduce with one child,
and another is chosen at random to die. The offspring of a parent of
type~$i$ does not mutate with probability $1-\mu$, or mutates in a
parent independent way to type~$j$ with probability $\mu p_j$, $j\in\{
1,\ldots,d\}$, where $|\mb{p}| = 1$.
The generator of the process is described by
\[
\mathcal{L}f(\mb{z}) = \sum_{i=1}^d\sum_{j=1}^d z_i\biggl(\frac{\lambda
}{N}z_j+\mu p_j \biggr) [f(\mb{z}-\mb{e}_i+\mb{e}_j) - f(\mb
{z})] .
\]
Setting $\alpha= M\mu\mb{p}/\lambda$, $\lambda= N/2$, the stationary
distribution of the process is ${ \operatorname{DM}}_{\alpha,N}$, the eigenvalues are
(\ref{geneig}), and the transition density is
\[
P\bigl(Z(t)=s\vert Z(0)=r\bigr)
= { \operatorname{DM}}_{\alpha,N}(s)\Biggl\{1 + \sum_{n=1}^N\rho_n(t)H_n(r,s)\Biggr\}.
\]
Thus (\ref{geneig}) is an example of positive-definite sequence
satisfying both (P1) and (P2).
\item[(ii)]
\[
\rho_{ n}(z)=z^{ n},\qquad  n=0,1,\ldots;
\]
that is, the eigenvalues of the so-called Poisson kernel, whose
positivity is a well-known result in special functions theory (see,
e.g.,~\cite{I05,DX02}).
\end{enumerate}
An interpretation of Poisson kernels as Markov transition semigroups is
in~\cite{GS10}, where it is shown that (ii) can be obtained via an
appropriate subordination of the genetic model~(i).

It is therefore natural to ask when (P1) holds with no constraints on
the parameter $\alpha.$

Our second approach to Lancaster's problem will answer, in part, this
question. This approach is heavily based on the probabilistic
interpretation (P\'{o}lya urns with random draws in common) of the
Jacobi and Hahn polynomial kernels shown in the first part of the
paper. We will prove in Proposition~\ref{pr:gspd} that, if $\{d_{
m}\dvt m=0,1,2,\ldots\}$ is a probability mass function (p.m.f.) on $\mbb
{Z}_+,$ then every positive-definite sequence $\{\rho_{ n}\}_{
n=0}^{\infty}$ of the form
%
%
\begin{equation}
\rho_{ n}=\sum_{ m= n}^{\infty}\frac{ m!\Gamma(|\alpha|+ m)}{( m-
n)!\Gamma(|\alpha|+ m+ n)}d_{ m},\qquad
  m=0,1,\ldots,\label{d2rhoproto}
\end{equation}
satisfies (P1) for every choice of $\alpha;$ therefore (P1) can be used
to model canonical correlations with respect to the Poisson--Dirichlet
distribution.

In Section~\ref{sec:pd2prob} we investigate the possibility of a
converse result, that is, will find a set of conditions on a JPD
sequence $\rho$ to be of the form (\ref{d2rhoproto}) for a p.m.f. $\{d_{
m}\}.$

As for Hahn positive-definite sequences and (P2), our results will be
mostly a consequence of Proposition~\ref{pr:KDM}, where we establish
the following representation of Hahn kernels as mixtures of Jacobi kernels:
\[
H^{\alpha}_{ n}(\mb{r},\mb{s})=\frac{{( N- n)!\Gamma(|\alpha|+ N+ n)}}{
N!\Gamma(|\alpha|+ N)}\mbb{E}[Q_{ n}^{\alpha}(\mb{X},\mb{Y}) |
\mb
{r},\mb{s}] ,\qquad  n=0,1,\ldots
\]
for every $N\in\mbb{Z}_+$ and $\mb{r},\mb{s}\in N\Delta_{(d-1)},$ where
the expectation on the right-hand side is taken with respect to
$D_{\alpha+\mb{r}}\otimes D_{\alpha+\mb{s}}$, that is, a product of
\textit{posterior} Dirichlet probability measures. A similar result was
proven by~\cite{GS08} to hold for individual Hahn polynomials as well.
The interpretation is again in terms of dependent P\'{o}lya sequences
with random elements in common.

We will also show (Proposition~\ref{prp:hintf}) that a discrete version
of (\ref{protoprop}) (but with the appearance of an extra coefficient)
holds for Hahn polynomial kernels.

Based on these findings, we will be able to prove in Section \ref
{sec:hpds} some results ``close to'' (P2): that JPDSs can
indeed be viewed as a map from HPDSs, and vice versa, but such
mappings, in general, are not the inverse of each other.

On the other hand, we will show (Proposition~\ref{bernpds}) that every
JPDS is in fact the limit of a sequence of (P2)-positive-definite
sequences.

Our final result on HPDSs is in Proposition~\ref{prp:p2}, where we
prove that if, for fixed $N$, $d^{(N)}=\{d^{(N)}_{ m}\}_{ m\in\mbb
{Z}_+}$ is a probability distribution such that $d^{(N)}_l=0$ for
$l>N$, then (P2) holds properly for the JPDS $\rho$ of the form (\ref
{d2rhoproto}). Such sequences also satisfy (P1) and admit
infinite-dimensional Poisson--Dirichlet (and Ewens's sampling
distribution) limits.

The key for the proof of Proposition~\ref{prp:p2} is provided by
Proposition~\ref{vdmk}, where we show the connection between our
representation of Hahn kernels and a kernel generalization of a product
formula for Hahn polynomials, proved by Gasper~\cite{GAS73} in 1973.
Proposition~\ref{vdmk} is, in our opinion, of some interest, even
independently of its application.

\subsection{Outline of the paper}\label{sec1.1}
The paper is organized as follows. Section~\ref{sec:distn} will
conclude this Introduction by recalling some basic properties and
definitions of the probability distribution we are going to deal with.
In Section
\ref{sec:k1} an explicit description of
$Q_{ n}^{\alpha}$ is given in terms of mixtures of products of
multinomial probability distributions arising from dependent P\'{o}lya
urns with random elements in common.
We will next obtain (Section~\ref{sec:k2}) an explicit representation for
$H^{\alpha}_{ n}$ as \textit{posterior mixtures} of
$Q_{ n}^{\alpha}$. 
In the same section we will generalize Gasper's product formula to an
alternative representation of $H^\alpha_{ n}$ and will describe the
connection coefficients in the two representations.
In Sections~\ref{sec:kPD}--\ref{sec:esfk}, we will then show that
similar structure and probabilistic descriptions also hold for kernels
with respect to the
ranked versions of $D_{\alpha}$ and $\operatorname{DM}_{\alpha,N}$, and to
their infinite-dimensional limits, known as the Poisson--Dirichlet
and Ewens's sampling distribution, respectively. 
This will conclude the first part.

Sections~\ref{sec:irjpk}--\ref{sec:irhpk} will be the bridge between
the first and the second part of the paper. We will prove identity
(\ref
{protoprop}) for the Jacobi product formula and its Hahn equivalent. We
will point out the connection between (\ref{protoprop}) and another
multivariate Jacobi product formula due to Koornwinder and Schwartz
\cite{KS91}.

In Section~\ref{sec:pdandk} we will focus more closely on
positive-definite sequences (canonical correlations). We will use
results of Section~\ref{sec:irjpk} (first approach) to characterize
sequences obeying to (P1), with constraints on $\alpha.$

In Section
\ref{sec:prob} we will use a second probabilistic approach to find
sufficient conditions for (P1) to hold with no constraints on the
parameters, when a JPDS can be expressed as a linear functional of a
probability distribution on $\mbb{Z}_+.$ Every such sequence will be
determined by a probability mass function on the integers. We will
discuss the possibility of a converse mapping from JPDSs to probability
mass functions in Section~\ref{sec:pd2prob}.

In the remaining sections we will investigate the existence of
sequences satisfying (P2). In particular, in Section~\ref{sec:p2} we
will make a similar use of probability mass functions to find
sufficient conditions for a proper version of (P2).

\subsection{Elements from distribution theory}\label{sec:distn}
We briefly list the main definitions and properties of the probability
distributions that will be used in the paper. We also refer to \cite
{GS08} for further properties and related distributions. For $\alpha
,\mb
{x}\in\mbb{R}^d$ and $\mb{n}\in\mbb{Z}_+^d,$ denote
\[
\mb{x}^\alpha=x_1^{\alpha_1}\cdots x_d^{\alpha_d}, \qquad\Gamma
(\alpha)=\prod_{i=1}^{d}\Gamma(\alpha_i)
\]
and
\[
\pmatrix{{ n}\cr\mb{n}}=\frac{ n!}{\prod_{i=1}^dn_i!}.
\]
Also, we will use
\begin{eqnarray*}
(\mb{a})_{(\mb{x})}&=&\frac{\Gamma(\mb{a}+\mb{x})}{\Gamma(\mb
{a})},\\
(\mb{a})_{[\mb{x}]}&=&\frac{\Gamma(\mb{a}+\underline{1})}{\Gamma
(\mb
{a}+\underline{1}-\mb{x})},
\end{eqnarray*}
whenever the ratios are well defined. Here $\underline{1}:=(1,1,\ldots,1).$\vadjust{\goodbreak}

If $x\in\mbb{Z}_+$, then $(a)_{(x)}=a(a+1)\cdots(a+x-1)$ and
$(a)_{[x]}=a(a-1)\cdots(a-x+1).$ $\mbb{E}_{\mu}$ will denote the
expectation under the probability distribution $\mu. $ The subscript
will be omitted when there is no risk of confusion.
\begin{definition}
\begin{enumerate}[(ii)]
\item[(i)] The $\operatorname{Dirichlet} (\alpha)$ distribution, $\alpha\in\mbb
{R}_+^d$, on the $d$-dimensional simplex
\[
\Delta_{(d-1)}:=\{\mb{x}\in[0,1]^d\dvt |\mb{x}|=1\}
\]

is given by
\[
D_{\alpha}(\mathrm{d}\mb{x}):=\frac{\Gamma(|\alpha|)\mb{x}^{\alpha
-\underline
{1}}}{\Gamma(\alpha)}\mbb{I}\bigl(\mb{x}\in\Delta_{(d-1)}\bigr)\,\mathrm{d}\mb{x}.
\]
\item[(ii)] The Dirichlet multinomial $(\alpha,N)$ distribution,
$\alpha
\in\mbb{R}_+^d,N\in\mbb{Z}_+$ on the $(d-1)$-dimensional discrete simplex
\[
N\Delta_{(d-1)}:=\{\mb{m}\in\mbb{Z}^d_+\dvt |\mb{m}|=N\}
\]
is given by the probability mass function
%
%
\begin{equation}
\operatorname{DM}_{\alpha,N}(\mb{r})=\pmatrix{{N}\cr\mb{r} }\frac{(\alpha)_{(\mb
{r})}}{(|\alpha|)_{(N)}},\qquad \mb{r}\in N\Delta
_{(d-1)}.\label{dm2}
\end{equation}
\end{enumerate}
\end{definition}
%
\subsubsection{P\'{o}lya sampling distribution}
$\operatorname{DM}_{\alpha,N}$ can be thought as the moment formula (sampling
distribution) of $D_\alpha$,
\[
\mbb{E}_{D_{\alpha}}\left[\pmatrix{{N}\cr\mb{r} }\mb{X}^{\mb{r}}\right],
\]
so $\operatorname{DM}_{\alpha,N}$ can be interpreted as the probability distribution
of a sample of $N$ random variables in $\{1,\ldots,d\}$, which are
conditionally independent and identically distributed with common law~$\mb{X}$, the latter being a random distribution with distribution
$D_\alpha.$ The probability distribution of~$\mb{X}$, conditional on a
sample of $N$ such individuals, is, by Bayes's theorem, again Dirichlet
with different parameters
%
%
\begin{equation}
D_{\alpha+\mb{r}}(\mathrm{d}\mb{x})=\frac{{{N}\choose\mb{r} }\mb{x}^{\mb
{r}}}{\operatorname{DM}_{\alpha,N}(\mb{r})}D_\alpha(\mathrm{d}\mb{x}).
\label{posdir}
\end{equation}
As $N\to\infty,$ the measure $\operatorname{DM}_{\alpha,N}$ tends to $D_\alpha.$
The Dirichlet multinomial distribution can also be thought as the
distribution of color frequencies arising in a sample of size $N$ from
a $d$-color P\'{o}lya urn. This sampling scheme can be described as
follows: in an urn there are $|\alpha|$ balls of which $\alpha_i$ are
of color $i,i=1,\ldots,d$ (for this interpretation one may assume,
without loss of generality, that $\alpha\in\mbb{Z}_+^d$). Pick a ball
uniformly at random, note its color, then return the ball in the urn
and add another ball of the same color. The probability of the first
sample to be of\vadjust{\goodbreak} color $i$ is $\alpha_i/|\alpha|.$ After simple
combinatorics one sees that the distribution of the color frequencies
after~$M$ draws is $\operatorname{DM}_{\alpha,M}$. Conditional on having observed
$\mb
{r}$ as frequencies in the first $M$ draws, the probability
distribution of observing $\mb{s}$ in the next $N-M$ draws is
%
%
\begin{equation}
\operatorname{DM}_{\alpha+\mb{r},N-M}(\mb{s})=D_{\alpha,N-M}(\mb{s})\frac
{D_{\alpha+\mb
{s},M}(\mb{r})}{D_{\alpha,M}(\mb{r})}.\label{posdm}
\end{equation}
%
\subsubsection{Ranked frequencies and limit distributions}
\noindent Define the \textit{ranking function} $\psi\dvtx\mbb
{R}^{d}\rightarrow\mbb{R}^d$ as the function reordering the elements of
any vector $\mb{y}\in\mbb{R}^d$ in decreasing order. Denote its
image by
\[
\psi(\mb{y})=\mb{y}^{\downarrow}=(y^{\downarrow}_1,\ldots
,y^{\downarrow}_d).
\]
The ranked continuous and discrete simplex will be denoted by $\Delta
_{d-1}^{\downarrow}=\psi(\Delta_{d-1})$ and $ N\Delta
_{d-1}^{\downarrow
}=\psi(N\Delta_{d-1}),$ respectively.
\begin{definition}
The \textup{Ranked Dirichlet} distribution with parameter $\alpha\in
\mbb
{R}_+^d,$ is one with density, with respect to the $d$-dimensional
Lebesgue measure
\[
D^{\downarrow}_{\alpha}(\mb{x}):=D_{\alpha}\circ\psi^{-1}(\mb
{x}^\downarrow)=\frac{1}{d!}\sum_{\sigma\in S_d}D_{\alpha}(\sigma
\mb
{x}^{\downarrow}),\qquad   \mb{x}\in\Delta_{d-1}^\downarrow,
\label{rdir}
\]
where $S_d$ is the group of all permutations on $\{1,\ldots,d\}$ and
$\sigma\mb{x}=(x_{\sigma(1)},\ldots,x_{\sigma(d)}).$

Similarly,
\[
\operatorname{DM}_{\alpha,N}^{\downarrow}:=\operatorname{DM}_{\alpha,N}\circ\psi^{-1}
\]
defines the \textup{ranked Dirichlet multinomial} distribution.
\end{definition}

With a slight abuse of notation, we will use $D^{\downarrow}$
to indicate both the ranked Dirichlet measure and its density. Ranked
symmetric Dirichlet and Dirichlet multinomial measures can be
interpreted as distributions on random partitions.
For every $\mb{r}\in N\Delta_{(d-1)}$ let $\beta_j=\beta_j(\mb
{r})$ be
the number of elements in $\mb{r}$ equal to $j$ and $k(\mb{r})=\sum
\beta_j(\mb{r}) $ the number of strictly positive components of $\mb
{r}.$ Thus $\sum_{i=1}^{r}i\beta_i(\mb{r})=N.$

For each $\mb{x}\in\Delta_{(d-1)}$ denote the monomial symmetric
polynomials by
\[
[\mb{x},\mb{r}]_d:=\sum_{
i_1\neq\cdots\neq i_k\in\{1,\ldots,d\}^k}\prod_{j=1}^{k}x_{i_j}^{r_j},
\]
where the sum is over all $d_{[k]}$ subsequences of $k$ distinct
integers, and let $[\mb{x},\mb{r}]$ be its extension to $\mb{x}\in
\Delta_{\infty}.$
Take a collection $(\xi_1,\ldots,\xi_{N})$ of independent, identically
distributed random variables, with values in a space of $d$ ``colors'' $(d\leq\infty)$, and assume that $x_j$ is the common
probability of any $\xi_i$ of being of color $j$. The function $[\mb
{x},\mb{r}]_d$ can be interpreted as the probability distribution of
any such sample realization giving rise to $k(\mb{r})$ distinct values
whose \textit{unordered} frequencies count $\beta_1(\mb{r})$ singletons,
$\beta_2(\mb{r})$ doubletons and so on.

There is a bijection between $\mb{r}^{\downarrow}=\psi(r)$ and
$\beta
(\mb{r})=(\beta_1(\mb{r}),\ldots,\beta_{N}(\mb{r})),$ both maximal
invariant functions with respect to permutations of coordinates, both
representing partitions of $N$ in $k(\mb{r})$ parts. Note that $[\mb
{x},\mb{r}]_d$ is invariant too, for every $d\leq\infty$. It is well
known that, for every $\mb{x}\in\Delta_{d}^{\downarrow},$
%
%
\begin{equation}
\sum_{\mb{r}^{\downarrow}\in N\Delta^{\downarrow
}_{(d-1)}}\pmatrix{N\cr{\mb
{r}^{\downarrow}}}\frac{1}{\prod_{i\geq1}\beta_{i}(\mb
{r}^{\downarrow
})!}[x,\mb{r}^{\downarrow}]_d=1,
\label{sym}
\end{equation}
that is, for every $\mb{x}$,
\[
\pmatrix{N\cr{\mb{r}^{\downarrow}}}\frac{1}{\prod_{i\geq1}\beta
_{i}(\mb
{r}^{\downarrow})!}[\mb{x},\mb{r}^{\downarrow}]_d
\]
represents a probability distribution on the space of random partitions
of $N.$

For $|\alpha|>0,$ let $D_{|\alpha|,d}, \operatorname{DM}_{|\alpha|,N,d}$ denote the
Dirichlet and Dirichlet multinomial distributions with symmetric
parameter $(|\alpha|/d,\ldots,|\alpha|/d).$
Then
%
%
\begin{eqnarray}
\label{formrdm}&&\operatorname{DM}^{\downarrow}_{|\alpha|,N,d}(\mb{r}^{\downarrow})\nonumber\\
&&\quad=\mbb
{E}_{D^{\downarrow}_{|\alpha|,d}}\biggl\{\pmatrix{N\cr{\mb{r}^{\downarrow
}}}\frac{1}{\prod_{i\geq1}\beta_{i}(\mb{r}^{\downarrow})!}[\mb
{x}^{\downarrow},\mb{r}^{\downarrow}]_d\biggr\}
\\
&&\quad= d_{[k]}
\frac{r!}{\prod_1^N j!^{\beta_j}\beta_j!}
\cdot
\frac{\prod_1^r (|\alpha|/d)^{\beta_j}_{(j)}}
{(|\alpha|)_{(N)}}
\nonumber\\
&&\quad \displaystyle\mathop{\longrightarrow}_{d\rightarrow\infty}
\frac{r!}{\prod_1^r j^{\beta_j}\beta_j!}
\cdot
\frac{|\alpha|^{k}}
{(|\alpha|)_{( r)}}
:=\operatorname{ESF}_{|\alpha|}(\mb{r})
\label{esf}.
\end{eqnarray}

\begin{definition}
The limit distribution $\operatorname{ESF}_{|\alpha|}(\mb{r})$ in (\ref{esf}) is
called the \textit{Ewens sampling formula} with parameter $|\alpha|.$
\end{definition}

\paragraph*{\texorpdfstring{Poisson--Dirichlet point process \protect{\cite{K75}}}{Poisson--Dirichlet point process [19]}} Let
$Y^{\infty}=(Y_1,Y_2,\ldots)$ be the sequence of points of a
non-homogeneous point process with intensity measure
\[
N_{|\alpha|}(y)= |\alpha|y^{-1}\mathrm{e}^{-y}.
\]
The probability generating functional is
%
%
\begin{eqnarray}\label{pdmgf}
\mathcal{F}_{|\alpha|}(\xi)&=&\mbb{E}_{|\alpha|}\biggl(\exp\biggl\{\int\log
\xi(y)N_{|\alpha|}(\mathrm{d}y)\biggr\}\biggr)
\nonumber
\\[-8pt]
\\[-8pt]
\nonumber
&=&\exp\biggl\{|\alpha|\int
_{0}^{\infty}\bigl(\xi(y)-1\bigr)y^{-1}\mathrm{e}^{-y}\,\mathrm{d}y\biggr\}
\end{eqnarray}
for suitable functions $\xi\dvtx\mbb{R}\rightarrow[0,1].$ Then
$|Y^{\infty
}|$ is a $\operatorname{Gamma}(|\alpha|)$ random variable and is independent of the
sequence of ranked, normalized points
\[
X^{\downarrow\infty}=\frac{\psi(Y^{\infty})}{|Y^{\infty}|}.
\]

\begin{definition}
The distribution of $X^{\downarrow\infty},$ is called the
\textit{Poisson--Dirichlet} distribution with parameter $|\alpha|.$
\end{definition}

\begin{proposition}
\begin{enumerate}[(ii)]
\item[(i)] The $\operatorname{Poisson\mbox{--}Dirichlet} (|\alpha|)$ distribution on
$\Delta_{\infty}$ is the limit
\[
\operatorname{PD}_{|\alpha|}=\lim_{d\rightarrow\infty}D^{\downarrow}_{|\alpha|,d}.
\]
\item[(ii)]The relationship between $D_{\alpha}$ and $\operatorname{DM}_{\alpha,N}$ is
replicated by ESF, which arises as the (symmetric) moment formula for
the PD distribution,
%
%
\begin{equation}
\operatorname{ESF}_{|\alpha|,N}(r)=\mbb{E}_{\operatorname{PD}_{|\alpha|}}\biggl\{\pmatrix { r\cr{\mb
{r}^{\downarrow}}}\frac{1}{\prod_{i\geq1}\beta_{i}(\mb
{r}^{\downarrow
})!}[\mb{x},\mb{r}^{\downarrow}]\biggr\}, \qquad r\in N\Delta^{\downarrow}.
\label{esfmom}
\end{equation}
\end{enumerate}
\end{proposition}

\begin{pf}
If $\mb{Y}=(Y_1,\ldots,Y_d)$ is a collection of $d$ independent random
variables with identical distribution $\operatorname{Gamma}(|\alpha|/d,1),$ then their
sum $|\mb{Y}|$ is a $\operatorname{Gamma}(|\alpha|)$ random variable independent of
$\mb{Y}/|\mb{Y}|,$ which has distribution $D_{\alpha|,d}.$ The
probability generating
functional of $\mb{Y}$ is (\cite{G79})
%
%
\begin{eqnarray}\label{pdlim}
\mathcal{F}_{|\alpha|,d}(\xi)&=&\biggl(1+\int_0^\infty\bigl(\xi(y)-1\bigr)
\frac{|\alpha|}{d}\frac{y^{{|\alpha|}/{d}-1}\mathrm{e}^{-y}}{\Gamma({|\alpha|}/{d}+1)}\,\mathrm{d}y\biggr)^{d}
\nonumber
\\[-8pt]
\\[-8pt]
\nonumber
&\displaystyle\mathop{\rightarrow}_{d\rightarrow\infty}&
\mathcal{F}_{|\alpha|}(\xi) ,
\end{eqnarray}
which, by continuity of the ordering function $\psi$, implies that
\textit{if $X^{\downarrow d}$ has distribution $D^{\downarrow}_{|\theta|,d},$
then}
\[
X^{\downarrow d}\mathop{\rightarrow}^{\mathcal{D}}X^{\downarrow
\infty}.
\]
This proves (i).
For the proof of (ii) we refer to~\cite{G79}.
\end{pf}

\section{Polynomial kernels in the Dirichlet distribution}
\label{sec:k1}
The aim of this
section is to show that, for every fixed $d\in\mbb{N}$ and $\alpha
\in
\mbb{R}^d,$ the orthogonal polynomial kernels with respect to
$D_\alpha
$ can be constructed from systems of two dependent P\'{o}lya urns
sharing a fixed number of random elements in common.

Consider two P\'{o}lya urns $U_1$ and $U_2$ with identical initial
composition $\alpha$, and impose on them the constraint that {the first
$ m$ draws from $U_1$ are identical to the first $m$ draws from $U_2$}.
For $M\leq N$ sample $M+ m$ balls from $U_1$ and $N+ m$ balls from
$U_2$. At the end of the experiment, the probability of having observed
frequencies $\mb{r}$ and $\mb{s}$, respectively, in the $M$
unconstrained balls sampled from $U_1$ and in the $N$ ones from $U_2,$
is, by (\ref{posdm}),
%
%
\begin{eqnarray}\label{jointpolya}
&&\sum_{|\mb{l}|= m}\operatorname{DM}_{\alpha, m}(\mb{l})\operatorname{DM}_{\alpha+\mb
{l},M}(\mb
{r})\operatorname{DM}_{\alpha+\mb{l},N}(\mb{s})
\nonumber
\\[-8pt]
\\[-8pt]
\nonumber
&&\quad  = \operatorname{DM}_{\alpha,M}(\mb{r})\operatorname{DM}_{\alpha,N}(\mb{s})\xi^{H,\alpha}_{
m}(\mb{r},\mb{s}),
\end{eqnarray}
where
%
%
\begin{equation}
\xi^{H,\alpha}_{ m}(\mb{r},\mb{s})=\sum_{|\mb{l}|= m}\frac
{\operatorname{DM}_{\alpha
+\mb{s}, m}(\mb{l})\operatorname{DM}_{\alpha+\mb{r}, m}(\mb{l})}{\operatorname{DM}_{\alpha,
m}(\mb
{l})}.\label{hxi}
\end{equation}
As $N,M\rightarrow\infty$, if we assume $N^{-1}\mb{s}\to\mb
{x},M^{-1}\mb{r}\to\mb{y}$, we find that this probability distribution
tends to
\[
D_\alpha(\mathrm{d}\mb{x})D_\alpha(\mathrm{d}\mb{y})\xi^{\alpha}_{ m}({\mb
{x}},{\mb{y}}),
\]
where
%
%
\begin{eqnarray}
\xi_{ m}^{\alpha}(\mb{x},\mb{y}) &=& \sum_{|\mb{l}|= m}\pmatrix{ m\cr
\mb{l}}
\frac{ |\alpha|_{( m)} }
{ \prod_1^d{\alpha_i}_{(l_i)} } \prod_1^d(x_iy_i)^{l_i} \label
{xin}\\
&=&\sum_{|\mb{l}|= m}\frac{{{ m}\choose\mb{l}}\mb{x}^{\mb{l} }  {{
m}\choose\mb{l}}\mb{y}^{\mb{l}}}{\operatorname{DM}_{\alpha, |m|}(\mb{l})}.\label{xin2}
\end{eqnarray}
Notice that, because P\'{o}lya sequences are exchangeable (i.e., their
law is invariant under permutations of the sample coordinates), the
same formula (\ref{xin2}) would hold even if we only assumed that the
sequences sampled from $U_1$ and $U_2$ have in common any $ m$ (and not
necessarily the first~$ m$) coordinates.
\subsection{\texorpdfstring{Polynomial kernels for $d\geq2$}{Polynomial kernels for d>=2}}

We shall now prove the following:
\begin{proposition}
\label{pr:qn} For every $\alpha\in\mbb{R}^{d}_{+}$ and every
integer $
n,$ the $ n$th orthogonal polynomial kernel, with respect to
$D_{\alpha}$, is given by
%
%
\begin{equation}
Q_{ n}^{\alpha}(\mb{x},\mb{y}) = \sum_{ m=0}^{ n}a_{ n m}^{|\alpha
|} \xi
_{ m }^{\alpha}(\mb{x},\mb{y}), \label{Dpoly}
\end{equation}
where
%
%
\begin{equation}
a_{ n m}^{|\alpha|}=(|\alpha| + 2 n -1)(-1)^{ n- m}\frac{(|\alpha| +
m)_{( n-1)}}{ m!( n- m)!} \label{anm}
\end{equation}
form a lower-triangular, invertible system.
An inverse relationship is
%
%
\begin{equation}
\xi^\alpha_{ m} (\mb{x},\mb{y})= 1 + \sum_{ n=1}^{ m} \frac{( m)_{[
n]}}{(|\alpha| +| m|)_{( n)}} Q_{ n}^{\alpha}(\mb{x},\mb{y}).
\label{inverse}
\end{equation}
\end{proposition}
\begin{remark}
A first construction of the Kernel polynomials was given by
\cite{G79a}. We provide here a revised proof. Operators with a role
analogous to the function $\xi_{ m}$ have, later on, appeared in
different contexts, but with little emphasis on P\'{o}lya urns or on
the probabilistic aspects (\cite{R99,W06} are some examples). A
closer, recent result is offered in~\cite{P08} where a multiple
integral representation for square-integrable functions with respect to
Ferguson--Dirichlet random measures is derived in terms of P\'{o}lya urns.
\end{remark}

\begin{pf*}{Proof of Proposition \protect\ref{pr:qn}}
Let $\{Q_{\mb{n}}^\circ\}$ be a system of \textit{orthonormal} polynomials with respect to $D_{\alpha}$ (i.e., such
$\mbb
{E}({Q_{\mb{n}}^\circ}^{2})=1$). We need to show that,
for independent Dirichlet distributed vectors $X,Y$, if $ n, k \leq
m$, then
%
%
\begin{equation}
\mbb{E}(\xi^{\alpha}_{ m}(\mb{X},\mb{Y})Q_{\mb{n}}^\circ(\mb
{X})Q_{\mb{k}}^\circ(\mb{Y})) = \delta_{\mb{n}\mb{k}} \frac{(
m)_{[ n]}}{(|\alpha| + m)_{( n)}}. \label{maineq}
\end{equation}
If this is true, an expansion is therefore
%
%
\begin{eqnarray}
\xi_{ m}^{\alpha}(\mb{x},\mb{y}) &=& 1 + \sum_{ n=1}^{ m} \frac{
( m)_{[ n]} } { (|\alpha| + m)_{( n)} } \sum_{\{\mb{n}: |\mb{n}|=n\}}
Q^\circ_{\mb{n}}(\mb{x})Q^\circ_{\mb{n}}(\mb{y})
\nonumber
\\[-8pt]
\\[-8pt]
\nonumber
& = & 1 + \sum_{ n=1}^{ m} \frac{( m)_{[ n]}}{(|\alpha| +
m)_{( n)}} Q_{ n}^{\alpha}(\mb{x},\mb{y}).
\end{eqnarray}
Inverting the triangular matrix with $(m,n)$th element
\[
\frac{( m)_{[ n]}}{(|\alpha| + m)_{( n)}}
\]
gives (\ref{Dpoly}) from (\ref{inverse}). The inverse matrix is
triangular with $( m, n)$th element
\[
(|\alpha| + 2 n-1) (-1)^{ n- m} \frac{(|\alpha| +
m)_{( n-1)}}{ m!( n- m)!},\qquad n\geq m,
\]
and the proof will be complete.

\textit{Proof of} (\ref{maineq}). Write
%
\begin{equation}
\mbb{E}\Biggl(\prod_1^{d-1}{X_i}^{n_i}\xi^{\alpha}_{ m}(\mb{X},\mb{Y})
\Big\vert\mb{Y} \Biggr) = \sum_{\{\mb{l}: |\mb{l}| = m\}} \pmatrix{ m\cr\mb
{l}}\prod_1^{d}Y_i^{l_i}
\frac{\prod_1^{d-1}(l_i + \alpha_i)_{(n_i)}}{(|\alpha| + m)_{( n)}}.
\label{xxi}
\end{equation}
Expressing the last product in (\ref{xxi}) as
\[
\prod_1^{d-1}(l_i+\alpha_i)_{(n_i)} = \prod_1^{d-1}{l_i}_{[n_i]} +
\sum
_{\{\mb{k}:|\mb{k}| < |n|\}}b_{\mb{n}\mb{k}}\prod_1^{d-1}{l_i}_{[k_i]}
\]
for constants
$b_{\mb{n}\mb{k}}$,
shows that
%
%
\begin{equation}
\mbb{E}\Biggl(\prod_1^{d-1}{X_i}^{n_i}\xi^{\alpha}_{ m}(\mb{X},\mb
{Y})\Big\vert\mb{Y} \Biggr) = \frac{( m)_{[n]}}{(|\alpha| + m)_{(n)}}\prod
_1^{d-1}Y_i^{n_i} +
R_0(\mb{Y}).
\end{equation}
Thus if $ n \leq k \leq m$,
%
%
\begin{eqnarray}
\mbb{E}(\xi^{\alpha}_{ m}(\mb{X},\mb{Y})Q^\circ_{\mb{n}}(\mb
{X})\vert \mb{Y}) &=& \frac{( m)_{[ n]}}{(|\alpha|+ m)_{( n)}}
\sum_{\{\mb{k}:|\mb{k}|= n\}}a_{\mb{n}\mb{k}}\prod_1^{d-1}{Y}_i^{k_i}
+ R_1(\mb{Y})\nonumber\\
&=& \frac{( m)_{[ n]}}{(|\alpha|+ m)_{( n)}}Q^\circ_{\mb{n}}(\mb
{Y}) +
R_2(\mb{Y}),
\end{eqnarray}
where
\[
\sum_{\{\mb{k}:|\mb{k}|= n\}}a_{\mb{n}\mb{k}}\prod_1^{d-1}{X}_i^{k_i}
\]
are terms of leading degree $ n$ in $Q^\circ_{\mb{n}}(\mb{X})$ and
$R_j(\mb{Y})$, $j=0,1,2$, are polynomials of degree less than $ n$ in
$\mb{Y}$. Thus if $ n \leq
k \leq m$,
%
%
\begin{eqnarray}\label{proofo}
\mbb{E}(\xi^{\alpha}_{ m}(\mb{X},\mb{Y})Q_{\mb{n}}^\circ(\mb
{X})Q_{\mb{k}}^\circ(\mb{Y})) &=& \mbb{E}\biggl(Q_{\mb{k}}^\circ
(\mb{Y}) \biggl\{
\frac{( m)_{[ n]}}{(|\alpha|+ m)_{( n)}}Q^\circ_{\mb{n}}(\mb{Y}) +
R_2(\mb{Y}) \biggr\} \biggr)
\nonumber
\\[-8pt]
\\[-8pt]
\nonumber
&=& \frac{( m)_{[ n]}}{(|\alpha|+ m)_{( n)}}\delta_{\mb{n}\mb{k}}.
\end{eqnarray}
By symmetry, (\ref{proofo}) holds for all $\mb{n},\mb{k}$ such that $
n, k\leq m$.
\end{pf*}

\subsection{Some properties of the kernel polynomials}
\label{sec:2_2}

\subsubsection*{Particular cases}
\begin{eqnarray*}
Q^\alpha_0 &=&1, \\
Q^{\alpha}_1 &=& (|\alpha|+1)(\xi_1 -1)\\
&=& (|\alpha| + 1)\biggl(|\alpha|\sum_1^\mathrm{d}x_iy_i/\alpha_i - 1\biggr),
\\
Q^{\alpha}_2 &=& \tfrac{1}{2}(|\alpha| + 3)\bigl((|\alpha|+2)\xi_2 -2
(|\alpha|+1)\xi_1 + |\alpha|\bigr),
\end{eqnarray*}
where
\[
\xi_2 = |\alpha|(|\alpha|+1)\biggl( \sum_1^d(x_iy_i)^2\big/\alpha_i(\alpha
_i+1) + 2\sum_{i<j}x_ix_jy_iy_j\big/\alpha_i\alpha_j\biggr).
\]

\subsubsection*{The $j$th coordinate kernel}
A well-known property of Dirichlet measures is that, if $\mb{Y}$ is a
$\operatorname{Dirichlet}(\alpha)$ vector in $\Delta_{(d-1)}$, then its $j$th
coordinate $Y_j$ has distribution $D_{\alpha_j,|\alpha|-\alpha_j}.$
Such a property is reflected in the Jacobi polynomial kernels.
For every $d$, let $\mb{e}_j$ be the vector in $\mbb{R}^d$ with every
$i$th coordinate equal $\delta_{ij}$, $i,j=1,\ldots,d$. Then
%
%
\begin{equation}
\xi^{\alpha}_{ m}(\mb{y},\mb{e}_j)=\frac{(|\alpha|)_{(
m)}}{(\alpha
_j)_{( m)}}y_j^{ m},  \qquad   m\in\mbb{Z}_+, \mb{y}\in\Delta_{(d-1)}.
\label{xiej}
\end{equation}
In particular,
%
%
\begin{equation}
\xi^{\alpha}_{ m}(\mb{e}_j,\mb{e}_k)=\frac{(|\alpha|)_{(
m)}}{(\alpha
_j)_{( m)}}\delta_{jk}.
\label{xiejek}
\end{equation}
Therefore, for every $d$ and $\alpha\in\mbb{R}_+^d,$ (\ref{xiej}) implies
%
%
\begin{eqnarray}\label{qej}
Q_{ n}^{\alpha}(\mb{y},\mb{e}_j)&=&\sum_{ m=0}^{ n}a_{ n
m}^{|\alpha
|}\xi_{ m}^{\alpha}(\mb{e}_j,\mb{y})
=Q_{ n}^{\alpha_j,|\alpha|-\alpha_j}(y_j,1)
\nonumber
\\[-8pt]
\\[-8pt]
\nonumber
&=&\zeta_{ n}^{\alpha_j,|\alpha|-\alpha_j}R_{ n}^{\alpha_j,|\alpha
|-\alpha_j}(y_j),\qquad j=1,\ldots,d, \mb{y}\in\Delta_{(d-1)},\nonumber
\end{eqnarray}
where
%
%
\begin{eqnarray}
R_{n}^{\alpha,\beta}(x)&=&\frac{Q_n^{\alpha,\beta}(x,1)}{Q_n^{\alpha
,\beta
}(1,1)}
\nonumber
\\[-8pt]
\\[-8pt]
\nonumber
&=&{}_{2}F_1\left(\matrix{
-n, n+\theta-1\vspace*{2pt}\cr
\beta}
\bigg|  {1-x}\right),\qquad n=0,1,2,\ldots\label{2djac}, \theta=\alpha
+\beta,
\end{eqnarray}
are univariate Jacobi polynomials $(\alpha>0,\beta>0)$ normalized by
their value at 1 and
\[
\frac{1}{\zeta^{\alpha,\beta}_{ n}}:={\mbb{E}}[R_{ n}^{\alpha
,\beta
}(X)]^2.
\]
In (\ref{2djac}), ${}_{p}F_q, p,q\in\mbb{N},$ denotes the Hypergeometric function (see~\cite{AS}
for basic properties).

\begin{remark} For $\alpha,\beta\in\mbb{R}_{+},$ let $\theta
=\alpha
+\beta.$ It is known (e.g.,~\cite{GS08}, (3.25)) that
%
%
\begin{equation}\frac{1}{\zeta^{\alpha,\beta}_{ n}}=n!\frac
{1}{(\theta
+2n-1)\ifact{(\theta)}{n-1}}\frac{ \ifact{(\alpha)}n}{\ifact{(\beta)}n}.\label{zetan}
\end{equation}
On the other hand, for every $\alpha=(\alpha_1,\ldots,\alpha_d),$
%
%
\begin{equation}
\zeta_{ n}^{\alpha_j,|\alpha|-\alpha_j}=Q_{ n}^{\alpha}(\mb
{e}_j,\mb{e}_j)=
\sum_{ m=0}^{ n}a_{ n m}^{|\alpha|}\frac{\ifact{(|\alpha|)}{
m}}{\ifact
{(\alpha_j)}{ m}}.
\label{zetaid}
\end{equation}
%
\end{remark}

\subsubsection*{Addition of variables in $x$}
\label{sec:sums}

Let $A$ be a $d^\prime\times d$ ($d^\prime< d$) 0--1 matrix whose rows
are orthogonal. A known property of the Dirichlet distribution is that,
if $\mb{X}$ has distribution $D_{\alpha},$
then $A\mb{X}$ has a ${
D}_{A\alpha}$ distribution. \noindent Similarly, with some easy computation
\[
\mbb{E}\bigl(\xi^{\alpha}_{ m}(\mb{X},\mb{y})\vert A\mb{X}=a\mb{x}\bigr) =
\xi^{A\alpha}_{ m}(A\mb{X},A\mb{y}).
\]
One has therefore the following:
\begin{proposition}\label{prp:col}A representation for Polynomial
kernels in $D_{A{\alpha}}$ is
%
%
\begin{equation}
Q^{A\alpha}_{ n}(A\mb{x},A\mb{y}) = \mbb{E}[Q^\alpha_{ n}(\mb
{X},\mb{y})\vert  A\mb{X}=A\mb{x}].
\end{equation}
\end{proposition}

\begin{example}\label{ex:proj} For any $\alpha\in\mbb{R}^{d}$ and
$k\leq d$, suppose $A\mb{X} = (X_1+\cdots+ X_k,X_{k+1} + \cdots+ X_d)
= {X}^\prime$. Then, denoting
$\alpha'=\alpha_1+\cdots+\alpha_k$ and $\beta'=\alpha
_{k+1}+\cdots
+\alpha_{d},$ one has
\[
Q_{ n}^{A\alpha}({x}^\prime,{y}^\prime) =\zeta_{ n}^{\alpha',\beta'}
R_{ n}^{\alpha',\beta'}(x^\prime)R_{ n}^{\alpha',\beta'}(y^\prime)
= \mbb{E}[Q^\alpha_{ n}(\mb{X},\mb{y}) \vert  X^\prime= x^\prime
].
\]
\end{example}

\section{Kernel polynomials on the Dirichlet multinomial distribution}
\label{sec:k2} For the Dirichlet multinomial distribution, it is
possible to derive an explicit formula for the kernel polynomials by
considering that Hahn polynomials can be expressed as \textit{posterior}
mixtures of Jacobi polynomials; cf.~\cite{GS08}, Proposition~5.2. Let $
\{Q^\circ_{\mb{n}}(\mb{x})\}$ be a orthonormal polynomial set on
the Dirichlet, considered as functions of $(x_1,\ldots,x_{d-1})$.
Define, for $\mb{r}\in N\Delta_{(d-1)}$,
%
%
\begin{equation}
h^\circ_{\mb{n}}(\mb{r} ; N) = \int Q^\circ_{\mb{n}}(\mb{x})
D_{\alpha
+\mb{r} }(\mathrm{d}\mb{x}), \label{transform}
\end{equation}
then $\{h^\circ_{\mb{n}}\}$ is a system of multivariate orthogonal
polynomials with respect to $\operatorname{DM}_{\alpha, N}$ with constant of orthogonality
%
%
\begin{equation}
\mbb{E}_{\alpha, N}[{h^\circ_{\mb{n}}(\mb{R} ;N)}]^{2}=
\frac{ (N)_{[ n]} } { (|\alpha|+N)_{( n)}}. \label{hahncp}
\end{equation}
Note also that if $N \to
\infty$ with $r_i/N \to x_i$, $i=1,\ldots,d$, then
\[
\lim_{N\to\infty} h^\circ_{\mb{n}}(\mb{r} ; N) = Q^\circ_{\mb
{n}}(\mb{x}).
\]

\begin{proposition}
\label{pr:KDM} The Hahn kernel polynomials with respect to $\operatorname{DM}_{\alpha
,N}$ are
%
%
\begin{equation}
H^{\alpha}_{ n}(\mb{r},\mb{s}) = \frac{(|\alpha| + N)_{( n)}}{N_{[
n]}}\int\int Q^{\alpha}_{ n}(\mb{x},\mb{y}) D_{\alpha+\mb
{r}}(\mathrm{d}\mb
{x})D_{\alpha+\mb{s}}(\mathrm{d}\mb{y})\label{KDM}
\end{equation}
for $\mb{r} = (r_1,\ldots,r_d)$, $\mb{s} = (s_1,\ldots,s_d)$, $|\mb
{r}|= |\mb{s}|=N$ fixed, and $ n = 0,1,\ldots, N$.

An explicit expression is
%
%
\begin{equation}
H^{\alpha}_{ n}(\mb{r},\mb{s}) = \frac{ (|\alpha| + N)_{( n)}
}{ r_{[ n]}}\cdot
\sum_{m=0}^na_{ n m}^{|\alpha|}\xi^{H,\alpha}_{ m}(\mb{r},\mb{s}),
\label{explicitDM}
\end{equation}
where $(a_{ n m}^{|\alpha|})$ is as in (\ref{anm}) and $\xi
^{H,\alpha
}_{ m}(\mb{r},\mb{s})$ is given by (\ref{hxi}).
\end{proposition}

\begin{pf}
The kernel sum is, by definition,
%
%
\begin{equation}
H^{\alpha}_{ n}(\mb{r},\mb{s}) = \frac{ (|\alpha|+N)_{( n)} } {
N_{[ n]} } \sum_{\{\mb{n}: |\mb{n}| =n\}}
h^\circ_{\mb{n}}(\mb{r} ; N)h^\circ_{\mb{n}}(\mb{s} ; N),
\end{equation}
and from (\ref{KDM}), (\ref{explicitDM}) follows. The form of $\xi_{
m}^{H,\alpha}$ is obtained by taking the expectation of
$\xi_{ m}^{\alpha}(\mb{X},\mb{Y}),$ appearing in the representation
(\ref{Dpoly}) of $Q_{ n}^{\alpha},$ with respect to the product measure
$D_{\alpha+\mb{r}}D_{\alpha+\mb{s}}.$
\end{pf}

The first polynomial kernel is
\[
H^{\alpha}_1(\mb{r},\mb{s}) = \frac{(|\alpha|+1)(|\alpha
|+{r})}{|\alpha
|}
\Biggl(\frac{|\alpha|}{(|\alpha|+N)^2}
\sum_1^d\frac{(\alpha_i+r_i)(\alpha_i+s_i)}{\alpha_i} - 1 \Biggr).
\]

\noindent\textit{Projections on one coordinate}\vspace*{6pt}

\noindent As in the Jacobi case, the connection with Hahn polynomials on $\{
0,\ldots,N\}$ is given by marginalization on one coordinate.
\begin{proposition}
\label{l:chu} For $N\in\mbb{N}$ and $d\in\mbb{N}$, denote ${\mb
{r}}_{j,1}=N\mb{e}_j\in\mbb{N}^{d},$ where $\mb{e}_j=(0,\ldots
,0,1,0, \ldots,0)$ with $1$ only at the $j$th coordinate.

For every $\alpha\in\mbb{N}^d,$
%
%
\begin{equation}
H_{ n}^{\alpha}(\mb{s},N\mb{e}_j)=\frac{1}{c^{|\alpha|}_{N,
n}}{h}^{\circ(\alpha_j,|\alpha|-\alpha_j)}_{ n}(s_j;N)
{h}^{\circ(\alpha_j,|\alpha|-\alpha_j)}_{ n}(N;N),\qquad
|\mb{s}|=N, \label{mhk}
\end{equation}
where
\[
c^{|\alpha|}_{N, n}:=\frac{(N)_{[ n]}}{\ifact{(|\alpha|+N)}{
n}}=\mbb
{E}\bigl[{h}^{\circ(\alpha,\beta)}_{ n}(R;N)^2\bigr],
\]
and $\{{h}^{\circ(\alpha_j, |\alpha|-\alpha_j)}_{ n}\}$ are orthogonal
polynomials with respect to $\operatorname{DM}_{(\alpha_j,|\alpha|-\alpha_j), N}$.
\end{proposition}
\begin{pf}
Because for every $d$ and $\alpha\in\mbb{R}^{d}_+$
\[
H_{ n}^{\alpha}(\mb{s},\mb{r})=\frac{1}{c^{|\alpha|}_{N, n}}\sum_{
m=0}^{ n}a^{|\alpha|}_{ n m}\xi^{H,\alpha}_{ m}(\mb{r},\mb{s})
\]
for $d=2$ and
$\alpha,\beta>0$ with $\alpha+\beta= |\alpha|,$
\[
{h}_{ n}^{\circ(\alpha,\beta)}(k;N){h}_{ n}^{\circ(\alpha,\beta
)}(j;N)=\sum_{ m=0}^{ n}a^{|\alpha|}_{ n m}\xi_{ m}^{H,\alpha,\beta
}(k,j),\qquad k,j=0,\ldots,N.
\]
Now, for $\mb{r},\mb{s}\in N\Delta_{(d-1)}$, rewrite $\xi_{
m}^{H,\alpha
}$ as
%
%
\begin{eqnarray}
\xi_{ m}^{H,\alpha}(\mb{s},\mb{r})&=&\sum_{|\mb{l}|=
m}\frac{\operatorname{DM}_{\alpha+\mb{s}, m}(\mb{l})\operatorname{DM}_{\alpha+\mb{r}, m}(\mb
{l})}{\operatorname{DM}_{\alpha, m}(\mb{l})}
\nonumber
\\[-8pt]
\\[-8pt]
\nonumber
&=&\sum_{|\mb{l}|= m}\operatorname{DM}_{\alpha+\mb{s}, m}(\mb{l})\frac
{\operatorname{DM}_{\alpha+\mb
{l},N}(\mb{r})}{\operatorname{DM}_{\alpha,N}(\mb{r})}. \label{eqxih}
\end{eqnarray}
Consider, without loss of generality, the case $j=1$. Since, for every
$\alpha$,
\[
\operatorname{DM}_{\alpha, m}(\mb{l})=\operatorname{DM}_{(\alpha_1,|\alpha|-\alpha_1),
m}(l_1)\operatorname{DM}_{(\alpha_2,\ldots,\alpha_d), m-l_1}(l_2,\ldots,l_d),
\]
then\vspace*{-1pt}
%
%
\begin{eqnarray}
\xi_{ m}^{H,\alpha}(\mb{s},Ne_1)&=&\sum_{l_1=0}^{ m}\operatorname{DM}_{(\alpha
_1+s_1,|\alpha|-\alpha_1+ m-s_1), m}(l_1)\frac{\operatorname{DM}_{(\alpha
_1+l_1,|\alpha
|-\alpha_1+ m-l_1),N}(N)}{\operatorname{DM}_{(\alpha_1,|\alpha|-\alpha_1),N}(N)}
\nonumber\\[-0.5pt]
&& \hspace*{14pt}{}\times\sum_{|\mb{u}|= m-l_1}\operatorname{DM}_{\alpha'+\mb{s}',
m-l_1}(\mb{u})\frac{\operatorname{DM}_{\alpha+l,0}(0)}{\operatorname{DM}_{\alpha,0}(0)} \nonumber\\[-0.5pt]
&=&\sum_{l_1=0}^{ m}\operatorname{DM}_{\alpha+\mb{s}, m}(l_1)\frac{\operatorname{DM}_{\alpha
_1+l_1,N}(N)}{\operatorname{DM}_{\alpha,N}(N)}
\sum_{|\mb{u}|= m-l_1}\operatorname{DM}_{\alpha'+\mb{s}', m-l_1}(\mb{u}) \nonumber\\[-0.5pt]
&=&\sum_{l_1=0}^{ m}\operatorname{DM}_{\alpha+\mb{s}, m}(l_1)\frac{\operatorname{DM}_{\alpha
_1+l_1,N}(N)}{\operatorname{DM}_{\alpha,N}(N)}\label{redxi1}
\\[-0.5pt]
&=&\xi_{ m}^{H,\alpha_1,|\alpha|-\alpha_1}(s_1, N)\label{redxi2}.\vspace*{-1pt}
\end{eqnarray}
Then (\ref{mhk}) follows immediately.\vspace*{-1pt}
\end{pf}

\subsection{Generalization of Gasper's product formula for Hahn polynomials}
For $d=2$ and $\alpha,\beta>0$ the Hahn polynomials\vspace*{-1pt}
%
%
\begin{equation}
h^{\alpha,\beta}_{ n}(r;N)={}_{3}F_2\left(
\matrix{- n, n+\theta-1,-r\vspace*{2pt}\cr
\alpha,-N}
\bigg|  1\right), \qquad n=0,1,\ldots,N, \label{hahn}
\end{equation}
with $\theta=\alpha+\beta,$ have constant of orthogonality\vspace*{-1pt}
%
%
\begin{equation}
\frac{1}{{u}^{\alpha,\beta}_{N,n}}:=\mbb{E}_{\alpha,\beta}
[h^{\alpha,\beta}_{n}(R;N)]^{2}=\frac{1}{{N\choose
n}}\frac{\ifact{(\theta+N)}{n}}{\ifact{(\theta)}{n-1}}\frac
{1}{\theta
+2n-1}\frac{\ifact{(\beta)}{n}}{\ifact{(\alpha)}{n}}.\label{un}
\end{equation}
%
The following product formula was found by Gasper~\cite{GAS72}:\vspace*{-1pt}
%
%
\begin{eqnarray}\label{gasper}
&& h^{\alpha,\beta}_{ n}(r;N)h^{\alpha,\beta}_{ n}(s;N)
\nonumber
\\[-7pt]
\\[-7pt]
\nonumber
&&\quad=
\frac{(-1)^{ n}\ifact{(\beta)}{ n}}{\ifact{(\alpha)}{ n}}
\sum_{l=0}^{ n}\sum_{k=0}^{ n-l}\frac{(-1)^{l+k} n_{[l+k]}\ifact
{(\theta
+ n-1)}{l+k}r_{[l]}s_{[l]}(N-r)_{[k]}(N-s)_{[k]}}
{l!k!N_{[l+k]}N_{[l+k]}\ifact{(\alpha)}{l}\ifact{(\beta
)}{k}}.\qquad
\end{eqnarray}
Thus\vspace*{-1pt}
%
%
\begin{eqnarray}\label{n2}
&& {u}^{\alpha,\beta}_{N,n}h^{\alpha,\beta}_{ n}(r;N)h^{\alpha,\beta}_{
n}(s;N)\nonumber\\
&&\quad=\frac{N_{[ n]}}{(\theta+ N)_{( n)}}\sum_{ m=0}^{ n}\frac
{(-1)^{ n- m}
\ifact{(\theta)}{ n-1}\ifact{(\theta+ n-1)}{ m}(\theta+2 n-1)}{
m!( n-
m)!\ifact{(\theta)}{ m}}{\chi}^{H,\alpha,\beta}_{ m}(r,s)
\\
&&\quad=\frac{N_{[ n]}}{(\theta+ N)_{( n)}}\sum_{ m=0}^{ n}a_{ n
m}^{\theta
}{\chi}^{H,\alpha,\beta}_{ m}(r,s),
\nonumber
\end{eqnarray}
where
%
%
\begin{equation}{\chi}^{H,\alpha,\beta}_{ m}(r,s):=\sum_{j=0}^{
m}\frac
{1}{\operatorname{DM}_{(\alpha,\beta), m}(j)}\biggl[\frac{{{ m}\choose
j}r_{[j]}(N-r)_{[ m-j]}}{N_{[ m]}}\biggr]\biggl[\frac{{{ m}\choose
j}s_{[j]}(N-s)_{[ m-j]}}{N_{[ m]}}\biggr].
\label{hatxi}
\end{equation}

By uniqueness of polynomial kernels, we can identify the
connection coefficients between the functions $\xi$ and $\chi$:
\begin{proposition}
For every $ m, n\in\mbb{Z}_+,$ and every $r,s\in\{0,\ldots,N\},$
%
%
\begin{equation}
\xi_{ m}^{H,\alpha,\beta}(r,s)=\sum_{ l=0}^{ m}b_{ m l}\chi_{
l}^{H,\alpha,\beta}(r,s),
\label{xichi1}
\end{equation}
where
%
%
\begin{equation}
b_{ m l}=\sum_{ n= l}^{ m}\biggl(\frac{N_{[ n]}}{(\theta+ N)_{(
n)}}\biggr)^2\frac{m_{[ n]}}{(\theta+ m)_{( n)}}a^{\theta}_{ n
l}.\label{ccoefxc}
\end{equation}

\end{proposition}
\begin{pf}
From (\ref{explicitDM}),
%
%
\begin{equation}
{u}^{\alpha,\beta}_{N,n}h^{\alpha,\beta}_{ n}(r;N)h^{\alpha,\beta}_{
n}(s;N)=H^{\alpha,\beta}_{ n}(r,s)
=\frac{(\theta+ N)_{( n)}}{N_{[ n]}}\sum_{ m=0}^{ n}a_{ n m}^{\theta
}\xi_{ m}^{H,\alpha,\beta}(r,s).\label{n1}
\end{equation}
Since the array $A=(a_{ n m}^{\theta})$ has inverse
$C=A^{-1}$ with entries
%
%
\begin{equation}c^{\theta}_{ m n}=\biggl(\frac{m_{[ n]}}{(\theta+ m)_{(
n)}}\biggr),
\label{cnm}
\end{equation}
then equating (\ref{n1}) and (\ref{n2}) leads to
\begin{eqnarray*}
\xi_{ m}^{H,\alpha,\beta}&=&\sum_{ n=0}^{ m}c^{\theta}_{ m n}\frac{N_{[
n]}}{(\theta+ N)_{( n)}}H_{ n}^{\alpha,\beta} \\[-2pt]
&=&\sum_{ n=0}^{ m}c^{\theta}_{ m n}\biggl(\frac{N_{[ n]}}{(\theta+
N)_{( n)}}\biggr)^2\sum_{ l=0}^{ n}a_{ n l}^{\theta}\chi_{ l}^{H,\alpha
,\beta} \\[-2pt]
&=&\sum_{ l=0}^{ m}b_{ m l}\chi_{ l}^{H,\alpha,\beta}.\vspace*{-2pt}
\end{eqnarray*}
\upqed\end{pf}

The following corollary is then straightforward.\vspace*{-2pt}
\begin{corollary}
\[
\mbb{E}[\xi_{ m}^{H,\alpha,\beta}\chi_{ l}^{H,\alpha,\beta}]=
\mbb{E}[\xi_{ l}^{H,\alpha,\beta}\chi_{ m}^{H,\alpha,\beta}
]=\sum_{ n=0}^{ m\wedge l}\frac{ m_{[ l]} l_{[ n]}}{(\theta+ m)_{(
n)}(\theta+ l)_{( n)}}.\vspace*{-2pt}
\]
\end{corollary}

For every $\mb{r}\in N\Delta_{(d-1)}$ and $\mb{m}\in\mbb
{Z}^d_{+}$, define
\[
p_{\mb{m}}(\mb{r})=\prod_{i=1}^{d}(r_i)_{[m_i]}.
\]
Gasper's product formula (\ref{gasper}), or, rather, the representation
(\ref{n2}), has a multivariate extension in the following.\vspace*{-2pt}

\begin{proposition}\label{vdmk}For every $d$, $\alpha\in\mbb{R}_{+}^d$
and $N\in\mbb{Z}_+,$ the Hahn polynomial kernels admit the following
representation:
%
%
\begin{equation}H_{ n}^{\alpha}(\mb{r},\mb{s})=\frac{N_{[
n]}}{(|\alpha
| + N)_{( n)}}\sum_{ m=0}^{ n}a_{ n m}^{|\alpha|}{\chi}^{H,\alpha}_{
m}(\mb{r},\mb{s}),\qquad \mb{r},\mb{s}\in N\Delta
_{(d-1)},\
n=0,1,\ldots,
\label{n22}
\end{equation}
where
%
%
\begin{equation}{\chi}^{H,\alpha}_{ m}(\mb{r},\mb{s}):=\sum_{\mb
{l}:|\mb
{l}|= m}\frac{1}{\operatorname{DM}_{\alpha, m}(\mb{l})}\biggl(\frac{{{ m}\choose\mb
{l}}p_{\mb{l}}(\mb{r})}{N_{[ m]}}\biggr)\biggl(\frac{{{ m}\choose{\mb
{l}}}p_{\mb{l}}(\mb{s})}{N_{[ m]}}\biggr).\vadjust{\goodbreak}
\label{hatxi2}
\end{equation}
\end{proposition}

\begin{pf}
If we prove that, for every $ m$ and $ n$,
\[
\chi_{ m}^{H,\alpha}(\mb{r},\mb{s})=\sum_{ n=0}^{ m}\frac{c_{ m
n}^{|\alpha|}}{c_{ N n}^{|\alpha|}}H^{\alpha}_{ n}(\mb{r},\mb{s}),
\]
where $c_{i j}^{|\alpha|}$ are given by (\ref{cnm}) (independent of
$d$!), then the proof follows by inversion.

Consider the orthonormal multivariate Jacobi polynomials
$Q^{\circ}_{n}(\mb{x}).$
The functions
\[
h^{\circ}_{\mb{n}}(\mb{r}; N):=\int_{\Delta_{(d-1)}}Q^{\circ
}_{\mb
{n}}(\mb{x})D_{\alpha+\mb{r}}(\mathrm{d}\mb{x})
\]
satisfy the identity
%
%
\begin{equation}
\mbb{E}\left[h^\circ_{\mb{n}}(\mb{R} ; N)\pmatrix{{ m}\cr\mb{l}} p_{\mb
{l}}(\mb{R})\right]=N_{[ m]}h^\circ_{\mb{n}}(\mb{l} ; m)\operatorname{DM}_{\alpha,
m}(\mb{l} ),\qquad \mb{l}\in m\Delta_{(d-1)},\mb{n}\in\mbb
{Z}_+^{d}
\end{equation}
(\cite{GS10}, (5.71)).

Then for every fixed $\mb{s}$,
%
%
\begin{equation}
\mbb{E}[\chi_{ m}^{H,\alpha}(\mb{R},\mb{s}) h^\circ_{\mb{n}}(\mb
{R}; N)]=\sum_{ l= m}\pmatrix{{ m}\cr\mb{l} } \frac{p_{\mb{l}}(\mb
{s})}{N_{[ m]}}h_{\mb{n}}^{\circ}(\mb{l};   m),\label{temph1}
\end{equation}
so, iterating the argument, we can write
%
%
\begin{equation}
\mbb{E}[\chi_{ m}^{H,\alpha}(\mb{R},\mb{S})h^{\circ}_{\mb
{n}}(\mb
{R} ; N)h^{\circ}_{\mb{n}}(\mb{S} ; N)]={c_{ m n}}.
\end{equation}
Now, by uniqueness of the polynomial kernel,
\[
H^{\alpha}_{ n}(\mb{r},\mb{s})=\sum_{ n=0}^{\infty}\frac
{1}{c_{N,n}^{|\alpha|}}h^{\circ}_{\mb{n}}(r ; N)h^{\circ}_{\mb
{n}}(\mb
{s} ; N),
\]
therefore
\[
\chi_{ m}^{H,\alpha}(\mb{r},\mb{s})=\sum_{ n=0}^{ m}\frac{c_{ m
n}^{|\alpha|}}{c_{ N n}^{|\alpha|}}H^\alpha_{ n}(\mb{r},\mb{s}),
\]
and the proof is complete.
\end{pf}

 The connection coefficients between $\xi^{H,\alpha}_{ m}$ and
$\xi_{ m}^{\alpha}$ are, for every $d$, the same as for the
two-dimensional case:
\begin{corollary}For every $d$ and $\alpha\in\mbb{R}_+^d,$
\begin{enumerate}[(ii)]
\item[(i)]
%
%
\begin{equation}
\xi_{ m}^{H,\alpha,\beta}(\mb{r},\mb{s})=\sum_{ l=0}^{ m}b_{ m
l}\chi_{
l}^{H,\alpha,\beta}(\mb{r},\mb{s}),
\label{xichid}
\end{equation}
where $(b_{ m l})$ are given by (\ref{ccoefxc}).
\item[(ii)]

\[
\mbb{E}[\xi_{ m}^{H,\alpha}\chi_{ l}^{H,\alpha}]=
\mbb{E}[\xi_{ l}^{H,\alpha}\chi_{ m}^{H,\alpha}]=\sum_{
n=0}^{ m\wedge l}\frac{ m_{[ l]} l_{[ n]}}{(|\alpha|+ m)_{(
n)}(|\alpha
|+ l)_{( n)}},\qquad  m, l=0,1,2,\ldots.
\]
\end{enumerate}
\end{corollary}

\subsection{Polynomial kernels on the hypergeometric distribution}
Note that there is a direct alternative proof of orthogonality of
$H^{\alpha}_{ n}(\mb{r},\mb{s})$ similar to that for $Q^{\alpha}_{
n}(\mb{x},\mb{y})$. In the
Hahn analogous proof, orthogonality does not depend on the fact that
$|\alpha| >0$. In particular, we obtain \textit{kernels on the
hypergeometric distribution},
%
%
\begin{equation}
\frac{{c_1 \choose r_1}\cdots{c_d\choose r_d}}{{|c| \choose r}}
\end{equation}
by replacing $\alpha$ by $-c$ in (\ref{explicitDM}) and (\ref{hxi}).
Again a direct proof similar to that for $Q^{\alpha}_{ n}(\mb{x},\mb
{y})$ would be
possible.

\section{Symmetric kernels on ranked Dirichlet and Poisson--Dirichlet
measures} \label{sec:kPD}

From Dirichlet--Jacobi polynomial kernels we can also derive polynomial
kernels orthogonal with respect to symmetrized Dirichlet measures. Let
$D_{|\alpha|,d}$ be the Dirichlet distribution on
$d$ points with symmetric parameters
$(|\alpha|/d,\ldots,|\alpha|/d),$ and $D_{|\alpha|,d}^{\downarrow
}$ its
ranked version. Denote with $Q^{(|\alpha|,d)}_{ n}$ and $Q^{(|\alpha
|,d)\downarrow}_{ n}$ the corresponding $ n$-kernels.
\begin{proposition}
\label{pr:rjk}
\[
Q^{(|\alpha|,d)\downarrow}_{ n} = (d!)^{-1}\sum_\sigma Q^{(|\alpha
|,d)}_{ n}(\sigma(\mb{x}),\mb{y}),
\]
where summation is over all permutations $\sigma$ of $1,\ldots,d$. The
kernel polynomials have a
similar form to $Q^{(|\alpha|,d)}_{ n}$, but with $\xi^{(|\alpha
|,d)}_m$ replaced by
%
%
\begin{eqnarray}
\xi^{(|\alpha|,d)\downarrow}_{ m} &=& \sum_{\mb{l}\in m\Delta
^{\downarrow}_{(d-1)}} \frac{ m!|\theta|_{(m)}
(d-k)!(\prod_1^m \beta_ i(\mb{l})!)[\mb{x};\mb{l}][\mb{y};\mb
{l}] } { d! \prod_1^m
[j!(|\theta|/d)_{(j)}]^{\beta_j(\mb{l})} }\\
&=&\sum_{\mb{l}\in m\Delta^{\downarrow}_{(d-1)}}\frac{\sharp(\mb
{l})[\mb{x};\mb{l}]\sharp(\mb{l})[\mb{y};\mb{l}]}{\operatorname{DM}_{|\alpha|,
m,d}^{\downarrow}(\mb{l})}\label{symxi},
\end{eqnarray}
where
\[
\sharp(\mb{l}):=\pmatrix{ l\cr{\mb{l}}}\frac{1}{\prod_{i\geq1}\beta
_{i}(\mb{l})!}.
\]
%
\end{proposition}

\begin{pf}
Note that
%
%
\begin{eqnarray}
Q^{(|\alpha|,d)\downarrow}_{ n}(\mb{x},\mb{y})&=&\frac{1}{d!}\sum
_{\sigma\in\mathcal{G}_d}{Q}^{(|\alpha|,d)}_{ n}(\sigma
\mb{x},\mb{y})\nonumber\\
&=&\frac{1}{d!}\sum_{\sigma\in\mathcal{G}_d}\sum_{ m\leq n}a_{ n
m}^{|\alpha|}\xi^{(|\alpha|,d)}_{ m}(\sigma
\mb{x},\mb{y})\nonumber\\
&=&d!\sum_{ m\leq n}a_{ n m}^{|\alpha|}\frac{1}{(d!)^2}\sum_{\sigma
\in
\mathcal{G}_d}\sum_{|\mb{l}|= m}\frac{{ m\choose
\mb{l}}^2(\sigma\mb{x})^{\mb{l}}\mb{y}^{\mb{l}}}{\operatorname{DM}_{|\alpha|,
m,d}(\mb
{l})}\nonumber\\
&=&\sum_{ m\leq n}a_{ n m}^{|\alpha|}\frac{1}{(d!)^2}\sum_{\sigma
,\tau
\in\mathcal{G}_d}\sum_{|\mb{l}|= m}\frac{{ m\choose
\mb{l}}^2(\sigma\tau\mb{x})^{\mb{l}}(\mb{y})^{\mb
{l}}}{\operatorname{DM}_{|\alpha|,
m}(\mb{l})}\nonumber\\
&=&\sum_{ m\leq n}a_{ n m}^{|\alpha|}\frac{1}{(d!)^2}\sum_{\sigma
,\tau
\in\mathcal{G}_d}\sum_{|\mb{l}|= m}\frac{{ m\choose
\mb{l}}^2(\sigma\mb{x})^{\mb{l}}(\tau\mb{y})^{\mb
{l}}}{\operatorname{DM}_{|\alpha|,
m,d}(\mb{l})}\label{sym0}\\
&=&\frac{1}{(d!)^2}\sum_{\sigma,\tau\in\mathcal
{G}_d}{Q}^{(|\alpha
|,d)}_{ n}(\sigma
\mb{x},\tau
\mb{y})\label{sym1}.
\end{eqnarray}
Now,
%
%
\begin{eqnarray}
\mbb{E}_{D_{(|\alpha|,d)}^{\downarrow}}\bigl[Q^{(|\alpha|,d)\downarrow
}_{ n}(\mb{x},\mb{Y})Q^{(|\alpha|,d)\downarrow}_{ m}(\mb{z},\mb
{Y})\bigr]
&=&\frac{1}{d!}\sum_{\sigma\in\mathcal{G}_d}{Q}^{(|\alpha|,d)}_{
n}(\sigma
\mb{x},\mb{z})\delta_{ n m}
\nonumber
\\[-8pt]
\\[-8pt]
\nonumber
&=&Q^{(|\alpha|,d)\downarrow}_{ n}(\mb{x},\mb{z})\delta_{ n m},
\end{eqnarray}
and hence $Q^{(|\alpha|,d)\downarrow}_{ n}$ is the $ n$ polynomial
kernel with respect to $D_{(|\alpha|,d)}^{\downarrow}.$\vspace*{1pt}
The second part of the theorem, involving identity (\ref{symxi}), is
just another way of rewriting (\ref{sym0}).
\end{pf}

\begin{remark}The
first polynomial is $Q^{(|\alpha|,d)\downarrow}_{1}\equiv0$.
\end{remark}

\subsection{Infinite-dimensional limit} As $d \to\infty$, $\xi
^{(|\alpha|,d)\downarrow}_{ m} \to\xi^{(|\alpha|,\infty
)\downarrow}_{
m}$, with
%
%
\begin{eqnarray}
\xi^{(|\alpha|,\infty)\downarrow}_{ m} &=& |\alpha|_{( m)} \sum
\frac{
m!(\prod_1^m
b_ i!)[\mb{x};\mb{l}][\mb{y};\mb{l}]} {|\alpha
|^k[0!1!]^{b_1}\cdots
[(k-1)!k!]^{b_k}}\\
&=& \sum\frac{\sharp(\mb{l})[\mb{x};\mb{l}]{{ m}\choose{\mb
{l}}}\sharp
(\mb{l})[\mb{y};\mb{l}]}{\operatorname{ESF}_{|\alpha|}(\mb{l})}.\label{xiinf}
\end{eqnarray}

\begin{proposition}
\label{pr:pdk}
The $ n$-polynomial kernel with respect to the Poisson--Dirichlet point
process is given by
%
%
\begin{equation}
Q^{(|\alpha|,\infty)\downarrow}_{ n} =\sum_{ m=0}^{ n}a_{ n
m}^{|\alpha
|}\xi^{(|\alpha|,\infty)\downarrow}_{ m}. \label{PDk}
\end{equation}
\end{proposition}


The first polynomial is zero, and the second polynomial is
\[
Q_2^\infty= (F_1 - \mu)(F_2 - \mu)/\sigma^2,
\]
where
\[
F_1 = \sum_1^\infty x_{(i)}^2, \qquad F_2 = \sum_1^\infty y_{(i)}^2,
\]
and
\[
\mu= \frac{1}{1+|\alpha|},\qquad \sigma^2 = \frac{2|\alpha
|}{(|\alpha
|+3)(|\alpha|+2)(|\alpha|+1)^2}.
\]

%
%
%
%
%
%

\subsection{Kernel polynomials on the Ewens sampling
distribution}\label{sec:esfk}

The Ewens sampling distribution can be obtained as a
limit distribution from the unordered Dirichlet multinomial
distribution $\operatorname{DM}^{\downarrow}_{|\alpha|,N,d}$ as $d\to\infty$. The
proof of the following proposition can be obtained by the same
arguments used to prove Proposition~\ref{pr:rjk}.\looseness=1

\begin{proposition}
\label{pr:rhk}
\begin{enumerate}[(ii)]
\item[(i)] The polynomial kernels with respect to
$\operatorname{DM}^{\downarrow}_{|\alpha|,N,d}$ are of the same form as (\ref
{explicitDM}), but with $\xi_{ m}^{H,(|\alpha|,d)}$ replaced by
%
%
\begin{equation}
\xi_{ m}^{H,(|\alpha|,d)\downarrow}:=(d!)^{-1}\sum_\pi\xi_{
m}^{H,(|\alpha|,d)}(\pi(\mb{r}),\mb{s}).
\label{dhsymk}
\end{equation}
\item[(ii)] The kernel polynomials with respect to $\operatorname{ESF}_{|\alpha|}$ are
derived by considering the limit form $\xi_{ m}^{H,|\alpha|\downarrow}$
of $\xi_{ m}^{H,(|\alpha|,d)\downarrow}.$
This has the same
form as $\xi_{ m}^{|\alpha|\downarrow}$ (\ref{xiinf}) with $[\mb
{x};\mb
{b}][\mb{y};\mb{b}]$
replaced by $[\mb{r};\mb{b}]^\prime[\mb{s};\mb{b}]^\prime$, where
\[
[\mb{r};\mb{b}]^\prime= (|\alpha|+|\mb{r}|)^{-1}_{( m)} \sum
{r_{i_1}}_{(l_1)}\cdots{r_{i_k}}_{(l_k)},
\]
and summation is over $\sum_1^{ m}jb_j= m$, $\sum_1^{ m}b_j = k$,
$k=1,\ldots, m$. The kernel polynomials have the same form as
(\ref{explicitDM}) with $\xi^{H,(|\alpha|,d)}_{ m}$ replaced by
$\xi_{ m}^{H,|\alpha|\downarrow}$. The first polynomial is
identically zero
under this symmetrization.
\end{enumerate}
\end{proposition}
%

\section{\mbox{Integral representation for Jacobi polynomial kernels}}\label{sec:irjpk}\vspace*{-2pt}

This section and Section~\ref{sec:irhpk} are a bridge between the first
and the second part of the paper. We provide an integral representation
for Jacobi and Hahn polynomial kernels, extending to $d\geq2 $ the
well-known Jacobi and Hahn product formulae found by Koornwinder and
Gasper for $d=2$ (\cite{Koo74,GAS72} and~\cite{GAS73}). It will
be a key tool to identify, under certain conditions on the parameters,
positive-definite sequences on the discrete and continuous
multi-dimensional simplex. The relationship between our integral
representation and a $d$-dimensional Jacobi product formula due to
Koornwinder and Schwartz~\cite{KS91} is also explained (Section \ref
{sec:csprod}).\vspace*{-2pt}

\subsection{Product formula for Jacobi polynomials when $d=2$}\vspace*{-2pt}
For $d=2$, consider the shifted Jacobi polynomials normalized by their
value at 1,
%
%
\begin{equation}
R_{n}^{\alpha,\beta}(x)=\frac{Q_n^{\alpha.\beta}(x,1)}{Q_n^{\alpha
,\beta}(1,1)}.
\end{equation}
They can also be obtained from the ordinary
Jacobi polynomials $P_{n}^{a,b}$ $(a,b>-1)$ with Beta weight
measure
\[
w_{a,b}=(1-x)^{a}(1+x)^{b}\,\mathrm{d}x,\qquad x\in[-1,1]
\]
via the transformation
%
%
\begin{equation}R_{n}^{\alpha,\beta}(x)=\frac{P_{n}^{\beta-1,\alpha
-1}(2x-1)}{P_{n}^{\beta-1,\alpha-1}(1)}.
\label{01J}
\end{equation}
The constant of orthogonality $
{\zeta_{n}^{(\alpha,\beta)}}$ is given by (\ref{zetan}).

A crucial property of Jacobi polynomials is that, under
certain conditions on the parameters, products of Jacobi
polynomials have an integral representation with respect to a
positive (probability) measure. The following theorem is part of a more
general result of Gasper~\cite{GAS72}.
\begin{theorem}[(Gasper~\cite{GAS72})]\label{th:gasper}
A necessary and sufficient condition for the equality
%
%
\begin{equation}
\frac{P_{n}^{a,b}(x)}{P_{n}^{a,b}(1)}\frac
{P_{n}^{a,b}(y)}{P_{n}^{a,b}(1)}=\int_{-1}^{1}\frac
{P_{n}^{a,b}(z)}{P_{n}^{a,b}(1)}\wt{m}_{x,y;a,b}(\mathrm{d}z),
\label{-11p}
\end{equation}
to hold for a positive measure $\mathrm{d}\wt{m}_{x,y},$ is that $a\geq b>-1$,
and either $b\geq1/2$
or $a+b\geq0$. If $a+b>-1$ or if $a>-1/2$ and $a+b=-1$ with $x\neq
-y,$ then $\wt{m}_{x,y;a,b}$ is absolutely continuous with respect to
${w}_{a,b}$, with density of the form
%
%
\begin{equation}
\frac{\mathrm{d}\wt{m}_{x,y;a,b}}{\mathrm{d}{w}_{a,b}}(z)
=\sum_{n=0}^{\infty}\phi_n\frac
{P_{n}^{a,b}(x)}{P_{n}^{a,b}(1)}\frac
{P_{n}^{a,b}(y)}{P_{n}^{a,b}(1)}\frac{P_{n}^{a,b}(z)}{P_{n}^{a,b}(1)},
\label{3k}
\end{equation}
with $\phi_n={P_{n}^{a,b}}(1)^{2}/\mbb{E}[P_{n}^{a,b}(X)].$\vadjust{\goodbreak}
\end{theorem}

An explicit formula for the density (\ref{3k}) is possible when $a\geq
b>-1/2$.

%
\begin{equation}
\frac{P_{n}^{a,b}(x)}{P_{n}^{a,b}(1)}\frac{P_{n}^{a,b}(y)}{P_{n}^{a,b}(1)}
=
\int_{0}^{1}\int_{0}^{\uppi}
\frac{P_{n}^{a,b}(\psi)}{P_{n}^{a,b}(1)} \wt{m}_{a,b}(\mathrm{d}u,\mathrm{d}\omega),
\label{A9.15.8}
\end{equation}
where
\[
\psi(x,y;u,\omega)= \{(1+x)(1+y)+(1-x)(1-y)\}/2 + u\cos\omega
\sqrt{(1-x^2)(1-y^2)}- 1
\]
and
%
%
\begin{equation}
\wt{m}_{a,b}(\mathrm{d}u,\mathrm{d}\omega) =\frac{ 2\Gamma(a + 1) } {
\sqrt{\uppi}\Gamma(a - b)\Gamma(b + {1}/{2}) } (1-u^2)^{a - b
- 1}u^{2b +1}(\sin\omega)^{2b}\,\mathrm{d}u\,\mathrm{d}\omega. \label{expf}
\end{equation}
See~\cite{Koo74} for an analytic proof of this formula. Note that
$\phi(1,1;u,\omega) = 1$, so $\mathrm{d}\wt{m}_{a,b}(u,\omega)$ is a
probability measure.

Gasper's theorem can be rewritten in an obvious way, in terms
of the shifted Jacobi polynomials $R_{n}^{\alpha,\beta}(x)$ on
$[0,1]$:\vspace*{-2pt}

\begin{corollary}
\label{cor:1} For $\alpha,\beta>0$ the product formula
%
%
\begin{equation}
R_{n}^{\alpha,\beta}(x)R_{n}^{\alpha,\beta}(y)=\int
_{0}^{1}R_{n}^{\alpha
,\beta}(z){m}_{x,y;
\alpha,\beta}(\mathrm{d}z) \label{01pf}
\end{equation}
holds for a positive measure ${m}_{x,y; \alpha,\beta}$, if and
only if $\beta\geq\alpha$, and either $\alpha\geq1/2$ or
$\alpha+\beta\geq2$. In this case
$m_{x,y}^{(\alpha,\beta)}=\wt{m}_{2x-1,2y-1;\beta-1,\alpha-1}$
where $\mathrm{d}\wt{m}$ is defined by (\ref{3k}). The measure is absolutely
continuous if $\alpha+\beta\geq2$ or if $\beta>1/2$ and $\alpha
+\beta
>1$ with $x\neq y.$ In this case
\[
m_{x,y}^{(\alpha,\beta)}(\mathrm{d}z)=K(x,y,z)D_{\alpha,\beta}(\mathrm{d}z),
\]
where
%
%
\begin{equation}
K(x,y,z)=\sum_{n=0}^{\infty}\zeta_n^{\alpha,\beta}R_{n}^{\alpha
,\beta
}(x)R_{n}^{\alpha,\beta}(y)R_{n}^{\alpha,\beta}(z)\geq0.
\label{kfunc}\vspace*{-2pt}
\end{equation}
\end{corollary}

\begin{remark}
When $\alpha,\beta$ satisfy the constraints of Corollary~\ref{cor:1},
we will say that $\alpha,\beta$ \textit{satisfy Gasper's conditions}.\vspace*{-2pt}
\end{remark}

When $\alpha\geq1/2,$ an explicit integral identity
follows from (\ref{A9.15.8})--(\ref{expf}). Let
$m_{\alpha\beta}(\mathrm{d}u,\mathrm{d}\omega) = \wt{m}_{\beta-1,\alpha
-1}(\mathrm{d}u,\mathrm{d}\omega).$ Then
%
\begin{equation}
R_{n}^{\alpha,\beta}(x)R_{n}^{\alpha,\beta}(y)= \int_0^1\int
_0^\uppi
R_{n}^{\alpha,\beta}(\varphi) m_{\alpha\beta}(\mathrm{d}u,\mathrm{d}\omega),
\label{A9.15.8a}
\end{equation}
where for $x,y\in[0,1]$
%
%
\begin{equation}
\varphi(x,y;u,\omega)= xy+(1-x)(1-y) + 2u\cos\omega\sqrt
{x(1-x)y(1-y)}. \label{scalephi}
\end{equation}
In $\phi$ set $x \leftarrow2x-1, y \leftarrow2y-1$ to obtain (\ref{scalephi}).\vadjust{\goodbreak}

\subsection{Integral representation for $d>2$}
\label{sec:kint}

An extension of the product formula (\ref{01pf}) is possible for
the kernel $Q^{\alpha}_{n}$ for the bivariate Dirichlet of any
dimension $d.$
\begin{proposition}
\label{PROP:1} Let $\alpha\in\mathbb{R}_{+}^{d}$ such that, for
every $j=1,\ldots,d$, $\alpha_{j}\leq\sum_{i=1}^{j-1}\alpha_{i}$
and $1/2\leq\alpha_{j}$, or $\sum_{i=1}^{j}\alpha_{i}\geq2.$
Then, for every $\mb{x},\mb{y}\in\Delta_{(d-1)}$ and every
integer $ n,$
%
%
\begin{equation}
{Q_{ n}^{\alpha}(\mb{x},\mb{y})}=\mbb{E}\bigl[Q_{ n}^{\alpha_d, |\alpha
|-\alpha_d}(Z_d,1) |  \mb{x},\mb{y}\bigr],\label{k1}
\end{equation}
where, for every $\mb{x},\mb{y}\in\Delta_{(d-1)}$, $Z_d$ is the $[0,1]$
random variable defined by the recursion
%
%
\begin{equation}
Z_{1}\equiv1; \qquad Z_{j}=\Phi_{j}D_{j}Z_{j-1},\qquad j=2,\ldots,d, \label{zj}
\end{equation}
with
%
%
\begin{eqnarray}\label{x*j}
D_{j}&:=&\frac{(1-x_{j})(1-y_{j})}{(1-X^{*}_{j})(1-Y^{*}_{j})};\qquad
 X^{*}_{j}:=\frac{x_{j}}{1-x_{j}(1-\sqrt Z_{j-1})};
 \nonumber
 \\[-8pt]
 \\[-8pt]
 \nonumber
  Y^{*}_{j}&:=&\frac{y_{j}}{1-y_{j}(1-\sqrt Z_{j-1})},
\end{eqnarray}
where $\Phi_{j}$ is a random variable in $[0,1],$ with distribution
\[
\mathrm{d}m_{x^{*}_{j},y^{*}_{j};\alpha_{j},\sum_{i=1}^{j-1}\alpha_{i}},
\]
where $\mathrm{d}m_{{x},{y}; \alpha,\beta}$ is defined as in Corollary
\ref{cor:1}.
\end{proposition}

The proposition makes it natural to order the parameters
of the Dirichlet in a decreasing way, so that it is sufficient to
assume that $\alpha_{(1)}+\alpha_{(2)}\geq2$ to obtain the
representation~(\ref{k1}).

Since the matrix $A=\{a_{nm}\}$ is invertible, the proof of
Proposition~\ref{PROP:1} only depends on the properties of the function
$\xi$. The following lemma is in fact all we need.

\begin{lemma}
\label{LEMMA:1}
For every $ m\in\mbb{N},d=2,3,\ldots$ and
$\alpha\in\mathbb{R}^{d}$ satisfying the assumptions of
Proposition~\ref{PROP:1},   %
%
\begin{equation}
\xi^{\alpha}_{ m}(\mb{x},\mb{y})=\frac{|\alpha|_{( m)}}{(\alpha_{d})_{(
m)}}\mbb{E}[ Z_{d}^{ m} |  \mb{x},\mb{y}],\label{pr1}
\end{equation}
where $Z_{d}$ is defined as in Proposition
\ref{PROP:1}.
\end{lemma}

Let $\theta=\alpha+\beta$. Assume the lemma is true. From
(\ref{A9.15.8a}) and (\ref{2xi1}) we know that, for every
$n=0,1,\ldots
$ and every $s\in[0,1]$,
\[
Q^{\alpha,\beta}_{ n}(s,1)=
\sum_{ m\leq n}a^{\theta}_{ n m}\frac{\ifact{(\theta)}{ m}}{\ifact
{\alpha}{ m}}s^{ m}.
\]
Thus from (\ref{pr1})
\begin{eqnarray*}
Q_{ n}^{\alpha}(\mb{x},\mb{y})&=&\mbb{E}\biggl[\sum_{ m\leq
n}a^{|\alpha|}_{ n m}\frac{\ifact{(|\alpha|)}{ m}}{\ifact{\alpha_{d}}{
m}}Z_{d}^{ m} \big| \mb{x},\mb{y}\biggr]\nonumber\\
&=& \mbb{E}\bigl[ Q^{\alpha_{d},|\alpha|-\alpha_{d}}_{ n}(Z_{d},1) |  \mb{x},\mb{y}\bigr],
\end{eqnarray*}
which is what is claimed in Proposition~\ref{PROP:1}.

Now we proceed with the proof of the lemma.

\begin{pf*}{Proof of Lemma \protect\ref{LEMMA:1}}
The proof is by induction.

If $d=2$, $x,y\in[0,1]$,

%
\begin{equation}
\xi^{(\alpha,\beta)}_{ m}(x,y)=\sum_{j=0}^{ m}\pmatrix{ m\cr
j}\frac{(\alpha+\beta)_{( m)}}{(\alpha)_{(j)}(\beta)_{(
m-j)}}(xy)^{j}{\bmom{x}{y}{ m-j}}.
\label{2xi}
\end{equation}
Setting $y=1$, the only positive addend in (\ref{2xi}) is the one
with $j= m$, so
%
\begin{equation}
\xi^{(\alpha,\beta)}_{ m}(x,1)=\frac{(\alpha+\beta)_{(
m)}}{(\alpha)_{(
m)}} z^{ m}. \label{2xi1}
\end{equation}
Therefore, if $\theta=\alpha+\beta$, from (\ref{A9.15.8a}) and
(\ref
{2xi1}), we conclude
%
\begin{eqnarray}\label{2xi2}
\xi^{\alpha,\beta}_{ m}(x,y)&=&\sum_{j=0}^{ m}\pmatrix{ m\cr
j}\frac{(\theta)_{( m)}}{(\alpha)_{(j)}(\beta)_{(
m-j)}}(xy)^{j}{\bmom
{x}{y}{ m-j}}
\nonumber
\\[-8pt]
\\[-8pt]
\nonumber
&=&\frac{(\theta)_{( m)}}{(\alpha)_{( m)}}\int
_{[0,1]} z^{ m}m_{x,y; \alpha,\beta}(\mathrm{d}z).
\end{eqnarray}
Thus the proposition is true for $d=2$.

To prove the result for any general $d>2$, consider
%
%
\begin{eqnarray}\label{decomp}
\xi^{\alpha}_{ m}(\mb{x},\mb{y})&=&\sum_{m_{d}=0}^{ m}\pmatrix{ m\cr
m_{d}}{(x_{d}y_{d})}^{m_{d}}\bmom{x_{d}}{y_{d}}{ m-m_{d}}\frac
{{(|\alpha
|)}_{ m}}{(\alpha_{d})_{(m_{d})}(|\alpha|-\alpha_{d})_{(
m-m_{d})}}\qquad
\nonumber
\\[-8pt]
\\[-8pt]
\nonumber
&&\phantom{\sum_{m_{d}=0}^{ m}}{}\times
\sum_{{\tilde{m}}\in\mathbb{N}^{d-1}:|\tilde{m}|= m-m_{d}}\pmatrix{{
m-m_{d}}\cr{\tilde{m}}}\ppr{(|\alpha|-\alpha_{d})}{
m-m_{d}}{d-1}{\alpha}{\tilde{m}}\copr{d-1}{\tilde{x}}{\tilde
{y}}{\tilde{m}},
\end{eqnarray}
where $\tilde{x}_{i}=\frac{x_{i}}{1-x_{d}}$, $\tilde{y}_{i}=\frac
{y_{i}}{1-y_{d}}$ $(i=1,\ldots,d-1)$.\vspace*{1pt}

Now assume the proposition is true for $d-1$. Then
the inner sum of (\ref{decomp}) has a representation like (\ref{pr1}),
and we can write
%
\begin{eqnarray}\label{decomp2}
\xi^{\alpha}_{ m}(\mb{x},\mb{y})&=& \sum_{m_{d}=0}^{ m}\pmatrix{ m\cr
m_{d}}{(x_{d}y_{d})}^{m_{d}}\bmom{x_{d}}{y_{d}}{ m-m_{d}}\nonumber\\
&&{}\hspace*{20pt}\times\frac
{{(|\alpha
|)}_{ m}}{(\alpha_{d})_{(m_{d})}(|\alpha|-\alpha_{d})_{(
m-m_{d})}}\qquad
\\
&&{}\hspace*{20pt}\times
\frac{(|\alpha|-\alpha_{d})_{( m-m_{d})}}{(\alpha_{d-1})_{(
m-m_{d})}}\mbb{E}[
Z_{d-1}^{ m-m_{d}}\vert\wt{\mb{x}},\wt{\mb{y}}],\nonumber
\end{eqnarray}
where the distribution of $Z_{d-1},$ given $\wt{\mb{x}},\wt{\mb{y}}$,
depends only on $\tilde{\alpha}={(\alpha_{1},\ldots,\alpha
_{d-1})}.$ Now,
set
\begin{eqnarray*}
\frac{X^{*}_{d}}{1-X^{*}_{d}}&=&\frac{x_{d}}{(1-x_{d}){\sqrt{Z_{d-1}}}};
\\
\frac{Y^{*}_{d}}{1-Y^{*}_{d}}&=&\frac{y_{d}}{(1-y_{d}){\sqrt{Z_{d-1}}}},
\end{eqnarray*}
and define the random variable
%
\begin{equation}
D_{d}:=\frac{(1-x_{d})(1-y_{d})}{(1-X^{*}_{d})(1-Y^{*}_{d})}. \label{eta}
\end{equation}
Then simple algebra leads to rewriting equation (\ref{decomp2}) as
%
\begin{eqnarray}\label{decomp3}
\xi^{\alpha}_{ m}(\mb{x},\mb{y})&=&
\mbb{E}\Biggl[\frac{|\alpha|_{( m)}(D_{d}Z_{d-1})^{ m}}{\ifact{(\alpha
_{d-1}+\alpha_{d})}{ m}}
\Biggl(\sum_{m_{d}=0}^{ m}\pmatrix{ m\cr
m_{d}}\frac{\ifact{(\alpha_{d-1}+\alpha_{d})}{ m}}{\ifact{(\alpha
_{d})}{m_{d}}
\ifact{(\alpha_{d-1})}{ m-m_{d}}}
\nonumber
\\[-8pt]
\\[-8pt]
\nonumber
&&\hspace*{115pt}{}\times (X^{*}_{d}X^{*}_{d})^{m_{d}}\bmom
{X^{*}_{d}}{Y^{*}_{d}}{ m-m_{d}}\Biggr) \Big|\mb{x},\mb{y}\Biggr].\qquad
\end{eqnarray}
Now the sum in (\ref{decomp3}) is of the form (\ref{2xi}), with
$\alpha=\alpha_{d-1}$, $\beta=\alpha_{d}$, with $m$ replaced by
$m-m_{d}$ and the pair $(x,y)$ replaced by
$(x^{*}_{d},y^{*}_{d})$. Therefore we can use equality
(\ref{2xi2}) to obtain
%
\begin{eqnarray}\label{cvd}
\xi^{\alpha}_{ m}(\mb{x},\mb{y})&=&\mbb{E}\biggl[\frac{(|\alpha|)_{(
m)}}{\ifact{(\alpha_{d})}{ m}}(D_{d}Z_{d-1})^{ m}\mbb{E}( \Phi
_{d}^{ m} | X^{*}_{d},Y^{*}_{d} ) \big|\mb{x},\mb{y}
\biggr]
\nonumber
\\[-8pt]
\\[-8pt]
\nonumber
&=&\frac{(|\alpha|)_{( m)}}{\ifact{(\alpha_{d})}{ m}}\mbb{E}
[Z_{d}^{ m} |  \mb{x},\mb{y}]
\end{eqnarray}
(the inner conditional expectation being a function of $Z_{d-1}$) so
the proof is complete.
\end{pf*}

\subsection{Connection with a multivariate product formula by
Koornwinder and Schwartz}\label{sec:csprod}

For the individual, multivariate Jacobi polynomials orthogonal
with respect to $D_{\alpha}\dvt\alpha\in\mbb{R}^d,$ a~product formula
is proved in~\cite{KS91}. For every $\mb{x}\in\Delta_{(d-1)}$,
$\alpha\in\mathbb{R}_{+}^{d}$ and
$\mb{n}=(n_{1},\ldots,n_{d-1})\dvt|\mb{n}|=n,$ these polynomials can be written
as
%
%
\begin{equation}
R_{\mb{n}}^{\alpha}(\mb{x})=\prod_{j=1}^{d-1}\biggl[R_{n_{j}}^{\alpha
_{j},E_{j}+2N_{j}}\biggl(\frac{x_{j}}{1-\sum_{i=1}^{j-1}x_{i}}
\biggr)\biggr]
\biggl(1-\frac{x_{j}}{1-\sum_{i=1}^{j-1}x_{i}}\biggr)^{N_j}
\label{mvopd},
\end{equation}
where $E_{j}=|\alpha|-\sum_{i=1}^{j}\alpha_{i}$ and
$N_{j}=n-\sum_{i=1}^{j}n_{i}$. The normalization is such that
$R_{n}^{\alpha}(e_d)=1, $ where $e_d:=(0,0,\ldots,1)\in\mbb{R}^d.$
For an account of such polynomials see also~\cite{GS08}.\vspace*{-2pt}

\begin{theorem}[(Koornwinder and Schwartz)]
\label{th:ks91} Let
$\alpha\in\mathbb{R}^{d}$ satisfy $\alpha_{d}>1/2$ and, for every
$j=1,\ldots,d$, $\alpha_{j}\geq\sum_{i=j+1}^{d}\alpha_{i}$. Then,
for every $\mb{x},\mb{y}\in\Delta_{(d-1)}$ there exists a positive
probability measure $\mathrm{d}m^*_{\mb{x},\mb{y};\alpha}$ such that, for every
$\mb{n}\in\mathbb{N}_{+}^{d}$,
%
%
\begin{equation}
R_{\mb{n}}^{\alpha}(\mb{x})R_{n}^{\alpha}(\mb{y})=\int_{\Delta_{(d-1)}}
R_{\mb{n}}^{\alpha}(\mb{z})m^{*}_{\mb{x},\mb{y};\alpha}(\mathrm{d}\mb{z})
\label{dpf}.
\end{equation}
\end{theorem}

Note that Theorem~\ref{th:ks91} holds for conditions on
$\alpha$ which are stronger than our Proposition~\ref{PROP:1}. This is
the price to pay for the measure $m^{*}_{\mb{x},\mb{y};\alpha}$ of
Koornwinder and Schwartz\vspace*{2pt} to have an explicit description (we omit it
here), extending (\ref{expf}). It is
possible to establish a relation between the measure
$m^{*}_{\mb{x},\mb{y};\alpha}(z)$ of Theorem~\ref{th:ks91} and the
distribution of $Z_d$
of Proposition~\ref{PROP:1}.\vspace*{-2pt}

\begin{proposition}
\label{prop:mg} Let $\alpha$ obey the conditions of Theorem
\ref{th:ks91}. Denote with $m_{\mb{x},\mb{y};\alpha}$ the probability
distribution of $Z_d$ of Proposition~\ref{PROP:1} and $m^{*}_{\mb
{x},\mb
{y};\alpha}$ the mixing measure in Theorem
\ref{th:ks91}. Then
\[
m^{*}_{\mb{x},\mb{y};\alpha}=m_{\mb{x},\mb{y};\alpha}.\vspace*{-2pt}
\]
\end{proposition}

\begin{pf}
From Proposition~\ref{PROP:1},
\[
Q_{ n}^{\alpha}(\mb{x},\mb{y})=\zeta_{ n}^{\alpha_{d},|\alpha
|-\alpha
_{d}}\mbb{E}\bigl(
R_{ n}^{\alpha_{d},|\alpha|-\alpha_{d}}(Z_{d})\mu_{\mb{x},\mb
{y};\alpha
}(Z_d)\bigr).
\]
Now, by uniqueness,
%
%
\begin{eqnarray}
Q_{ n}^{\alpha}(\mb{x},\mb{y})&=&\sum_{|\mb{m}|= n}Q_{\mb
{m}}^{\alpha
}(\mb{x})Q_{\mb{m}}^{\alpha}(\mb{y})
\nonumber
\\[-8pt]
\\[-8pt]
\nonumber
&=&\sum_{|\mb{m}|= n}\zeta_{\mb{m}}^{\alpha}R_{\mb{m}}^{\alpha
}(\mb
{x})R_{\mb{m}}^{\alpha}(\mb{y}),
\end{eqnarray}
where $\zeta^{\alpha}_{\mb{n}}:=\mbb{E}(R_{\mb{n}}^\alpha)^{-2}.$

So, by Theorem~\ref{th:ks91} and because $R_{\mb{n}}(e_d)=1,$
%
%
\begin{eqnarray}\label{pfmg2}
Q_{ n}^{\alpha}(\mb{x},\mb{y})&=&\int\biggl(\sum_{|\mb{m}|= n}\zeta
_{\mb
{m}}^{\alpha}R_{\mb{m}}^{\alpha}(\mb{z})\biggr)\,\mathrm{d}m^{*}_{\mb{x},\mb
{y};\alpha}(\mb{z})
\nonumber
\\[-8pt]
\\[-8pt]
\nonumber
&=&\int Q^{\alpha}_{ n}(\mb{z},e_{d}) \,\mathrm{d}m^{*}_{\mb{x},\mb{y};\alpha
}(\mb{z}),
\end{eqnarray}
where $Q_{\mb{n}}^{\alpha}$ are orthonormal polynomials. But we know that
\[
Q_{ n}^{\alpha}(\mb{z},e_{d})=\zeta_{ n}^{\alpha_{d},|\alpha
|-\alpha
_{d}} R_{ n}^{\alpha_{d},|\alpha|-\alpha_{d}}(z_{d}),
\]
so
%
%
\begin{eqnarray}\label{must}
Q_{ n}^{\alpha}(\mb{x},\mb{y})&=&\zeta_{ n}^{\alpha_{d},|\alpha
|-\alpha
_{d}}\mbb{E}\bigl(
R_{ n}^{\alpha_{d},|\alpha|-\alpha_{d}}(Z_{d})\mu_{x,y;\alpha
}(Z_d)\bigr)
\nonumber
\\[-8pt]
\\[-8pt]
\nonumber
&=&\zeta_{ n}^{\alpha_{d},|\alpha|-\alpha_{d}}\mbb{E}\bigl(
R_{ n}^{\alpha_{d},|\alpha|-\alpha_{d}}(Z_{d})\mu^*_{x,y;\alpha
}(Z_d)\bigr).
\end{eqnarray}
Thus both $\mu_{x,y;\alpha}(z)$ and $\mu^*_{x,y;\alpha}(z)$ have
the same Riesz--Fourier expansion
\[
\sum_{ n=0}^{\infty}Q_{ n}^{\alpha}(\mb{x},\mb{y})R_{ n}^{\alpha
_{d},|\alpha|-\alpha_{d}}(z),
\]
and this completes the proof.
\end{pf}

\section{Integral representations for Hahn
polynomial kernels}\label{sec:irhpk}

Intuitively it is easy now to guess that a discrete integral
representation for Hahn polynomial kernels, similar to that shown by Proposition
\ref{PROP:1} for Jacobi kernels, should hold for any \mbox{$d\geq2$}. We can
indeed use Proposition~\ref{PROP:1} to derive such a representation. We
need to reconsider formula~(\ref{transform}) for Hahn polynomial in the
following version:
%
\begin{equation}
\wt{h}_{\mb{n}}^{\alpha}(\mb{r};N):=\int R_{\mb{n}}^{\alpha}(\mb
{x})D_{\alpha+\mb{r}}(\mathrm{d}\mb{x})=\frac{h^{0}_{\mb{n}}(\mb
{r};N)}{\sqrt
{\zeta^\alpha_{\mb{n}}}}, \qquad  \mb{r}\in N\Delta_{(d-1)}, \label{wth}
\end{equation}
with the new coefficient of orthogonality
%
%
\begin{equation}\label{omega}
\frac{1}{\omega^{\alpha}_{\mb{n},N}}:=\mbb{E}[\wt{h}_{\mb
{n}}^{\alpha}(\mb{R};N)]^2
=\frac{N_{[ n]}}{(|\alpha|+ r)_{(
n)}}\frac{1}{\zeta_{\mb{n}}^{\alpha}}.
\end{equation}
Formula (\ref{wth}) is equivalent to 
%
%
\begin{equation}
R_{\mb{n}}^{\alpha}(\mb{x})=\frac{(|\alpha|+N)_{( n)}}{N_{[
n]}}\sum
_{|\mb{m}|=N}\wt{h}_{ n}^\alpha(\mb{m};N)\pmatrix{N\cr\mb{m}}\mb
{x}^{\mb{m}},
\qquad \alpha\in\mbb{R}^{d},\mb{x}\in\Delta_{(d-1)} \label{bbj}
\end{equation}
(see~\cite{GS08}, Section~5.2.1 for a proof).
\begin{proposition}
\label{prp:hintf} For $\alpha\in\mbb{R}^{d}$ satisfying the same
conditions as in Proposition~\ref{PROP:1}, a~representation for
the Hahn polynomial kernels is
%
%
\begin{eqnarray}\label{hint}
&& H^{\alpha}_{ n}(\mb{r},\mb{s})=\omega^{\alpha_d,|\alpha|-\alpha
_d}_{{n},N}\frac{(|\alpha|+N)_{( n)}}{N_{[ n]}}\mbb{E}_{\mb{r},\mb
{s}}\bigl[\wt{h}^{\alpha_d,|\alpha|-\alpha_d}_{ n}(K;N)\bigr],
\nonumber
\\[-8pt]
\\[-8pt]
\nonumber
&&\quad n\leq
|\mb{r}|= |\mb{s}|=N,\alpha\in\mbb{R}^d,
\end{eqnarray}
where the expectation is taken with respect to the measure
%
%
\begin{equation}
u_{\mb{r},\mb{s};\alpha}(k):=\int_{\Delta_{(d-1)}}\int_{\Delta_{(d-1)}}
\mbb{E}\biggl[\pmatrix { r\cr k}Z_d^k (1-Z_d)^{ r-k}  \big|  \mb{x},\mb
{y}\biggr] D_{\alpha+\mb{r}}(\mathrm{d}\mb{x})D_{\alpha+\mb{s}}(\mathrm{d}\mb{y}), \label{urs}
\end{equation}
where $Z_d,$ for every $\mb{x},\mb{y}$, is the random variable defined
recursively as in Proposition~\ref{PROP:1}.
\end{proposition}

\begin{pf} From (\ref{KDM}),
\begin{eqnarray*}
H^{\alpha}_{ n}(\mb{r},\mb{s}) = \frac{(|\alpha| + N)_{( n)}}{N_{[
n]}}\times \int_{\Delta_{(d-1)}}\int_{\Delta_{(d-1)} } Q_{ n}^{\alpha
}(\mb
{x},\mb{y})D_{\alpha+\mb{r}}(\mathrm{d}\mb{x})D_{\alpha+\mb{s}}(\mathrm{d}\mb{y}).
\end{eqnarray*}
Then (\ref{k1}) implies
\begin{eqnarray*}
H^{\alpha}_{ n}(\mb{r},\mb{s})&=&\frac{\zeta^{\alpha_{d},|\alpha
|-\alpha
_{d}}_{ n}(|\alpha| + N)_{( n)}}{N_{[ n]}}\\
&&{}\times \int_{\Delta_{(d-1)}}\int
_{\Delta_{(d-1)} }\int_0^1
{R^{\alpha_{d},|\alpha|-\alpha_{d}}_{ n}(z_{d})}m_{\mb{x},\mb
{y};\alpha
}(\mathrm{d}z_d)D_{\alpha+\mb{r}}(\mathrm{d}\mb{x})D_{\alpha+\mb{s}}(\mathrm{d}\mb{y}),
\end{eqnarray*}
so, by (\ref{bbj}),
\begin{eqnarray*}
H^{\alpha}_{ n}(\mb{r},\mb{s})&=&\zeta^{\alpha_{d},|\alpha
|-\alpha
_{d}}_{ n}\biggl(\frac{(|\alpha|
+ N)_{( n)}}{N_{[ n]}}\biggr)^2\\
&&{}\times
\sum_{k\leq N}\wt{h}^{\alpha_{d},|\alpha|-\alpha_{d}}_{
n}(k;N)\int
_{\Delta_{(d-1)}}\int_{\Delta_{(d-1)} } \int_0^1\pmatrix{N\cr
k} z_d^k(1-z_d)^{N-k}\\
&&\hspace*{170pt}{}\times m_{\mb{x},\mb{y};\alpha}(\mathrm{d}z_d)D_{\alpha+\mb{r}}(\mathrm{d}\mb{x})D_{\alpha
+\mb
{s}}(\mathrm{d}\mb{y}),
\end{eqnarray*}
and the proof is complete.
\end{pf}

\section{Positive-definite sequences and polynomial kernels}\label{sec:pdandk}
We can now turn our attention to the problem of identifying
and possibly characterizing positive-definite sequences with respect to
the Dirichlet or Dirichlet multinomial probability distribution. We
will agree with the following definition which restricts the attention
to multivariate positive-definite sequences $\{\rho_{\mb{n}}\dvt\mb
{n}\in
\mbb{Z}^d_+\}$, which depend on $\mb{n}$ only via $ |\mb{n}|.$
\begin{definition}
For every $d\geq2$ and $\alpha\in\mbb{R}_+^d$, call a sequence $\{
\rho
_{ n}\}_{ n=0}^{\infty}$ an $\alpha$-Jacobi positive-definite sequence
($\alpha$-JPDS) if $\rho_0=1$ and, for every $\mb{x},\mb{y}\in
\Delta_{(d-1)},$
%
%
\begin{equation}
p(\mb{x},\mb{y})=\sum_{ n=0}^{\infty}\rho_{ n}Q^\alpha_{ n}(\mb
{x},\mb
{y})\geq0.
\label{jpds}\vadjust{\goodbreak}
\end{equation}
For every $d\geq2,$ $\alpha\in\mbb{R}_+^d$ and $N\in\mbb{Z}_+,$
call a
sequence $\{\rho_{ n}\}_{ n=0}^{\infty}$ an $(\alpha,N)$-Hahn
positive-definite sequence ($(\alpha,N)$-HPDS) if $\rho_0=1$ and, for
every $\mb{r},\mb{s}\in N\Delta_{(d-1)},$
%
%
\begin{equation}
p^H(\mb{r},\mb{s})=\sum_{ n=0}^{\infty}\rho_{ n}H_{ n}(\mb{r},\mb
{s})\geq0.
\label{jpds2}
\end{equation}
\end{definition}

\subsection{Jacobi positivity from the integral representation}\label
{sec:pdandk1}
A consequence of the product formulae (\ref{01pf}) and
(\ref{A9.15.8a}) is a characterization of positive-definite
sequences for the Beta distribution.

The following is a $[0,1]$-version of a theorem proved
by Gasper with respect to Beta measures on $[-1,1]$.

\begin{theorem}[(Bochner~\cite{B54}, Gasper~\cite{GAS72})]
\label{th:2} Let $D_{\alpha,\beta}$ be the Beta distribution on
$[0,1]$ with \mbox{$\alpha\leq\beta$}. If either $1/2\leq\alpha$ or
$\alpha+\beta\geq2$, then a sequence $\rho_{n}$ is
positive-definite for $D_{\alpha,\beta}$ if and only if
%
%
\begin{equation}
\rho_{n}=\int R_{n}^{\alpha,\beta}(z)\nu_{\alpha,\beta}(z)
\label{bpd}
\end{equation}
for a positive measure $\nu$ with support on $[0,1].$ Moreover, if
\[
u(x)=\sum_{n=0}^{\infty}\zeta_{n}^{\alpha,\beta}\rho_nR_n(x)\geq0
\]
with
\[
\sum_{n=0}^{\infty}\zeta_{n}^{\alpha,\beta}|\rho_n|<\infty,
\]
then
%
%
\begin{equation}\label{absu}
\nu(A)=\int_Au(x)D_{\alpha,\beta}(\mathrm{d}x)
\end{equation}
for every Borel set $A\subseteq[0,1].$
\end{theorem}

We refer to~\cite{B54,GAS72} for the technicalities
of the
proof. To emphasize the key role played by (\ref{01pf}), just
observe that the positivity of $\nu$ and (\ref{bpd}) entails the
representation
\[
p(x,y):=\sum_{n=0}^{\infty}\zeta_n\rho_nR^{\alpha,\beta
}_{n}(x)R^{\alpha
,\beta}_{n}(y)=\int_0^1u(z)m_{x,y;\alpha,\beta}(\mathrm{d}z)\geq
0,
\]
and $u(z)=p(z,1),$
whenever $u(1)$ is absolutely convergent.

To see the full extent of the
characterization, we recall, in a lemma, an
important property of Jacobi polynomials, namely, that two different
systems of Jacobi polynomials are connected by an integral formula
if their parameters share the same total sum.

\begin{lemma}
\label{l:bor} For $\mu>0$,\vspace*{-2pt}
%
%
\begin{equation}
\int_{0}^{1}{R}_{n}^{\alpha,\beta}\bigl(1-(1-x)z\bigr) D_{\beta,\mu}(\mathrm{d}z)={R}_{n}^{\alpha-\mu,\beta+\mu}(x)
\label{01borr1}\vspace*{-2pt}
\end{equation}
and\vspace*{-2pt}
%
%
\begin{equation}
\int_{0}^{1}{R}_{n}^{\alpha,\beta}(xz) D_{\alpha,\mu}(\mathrm{d}z)=\frac{\zeta_n^{\alpha+\mu,\beta-\mu}}{\zeta
_{n}^{\alpha,\beta}}{R}_{n}^{a+\mu,b-\mu}(x)
\label{01borr2}.\vspace*{-2pt}
\end{equation}
\end{lemma}

\begin{pf}
We provide here a probabilistic proof in terms of polynomial
kernels $Q_{ n}^{\alpha,\beta}(x,y)$, even though the two
integrals can also be view as a reformulation, in terms of the
shifted polynomials $R_{n}^{\alpha,\beta},$ of known integral
representations for the Jacobi polynomials $\{P^{a,b}_n\}$ on
$[-1,1]$ $(a,b>-1)$ (see, e.g.,~\cite{AAR99} ff. formulae~7.392.3 and formulae~7.392.4).

Let us start with (\ref{01borr2}). The moments of a $\operatorname{Beta}
(\alpha,\beta)$ distribution on $[0,1]$ are, for every integer
$m\leq n=0,1,\ldots$\vspace*{-2pt}
\[
\mbb{E}[X^{m}(1-X)^{n-m}]=\frac{\ifact{\alpha}{m}\ifact{\beta
}{n-m}}{\ifact{(\alpha+\beta)}{n}}.\vspace*{-2pt}
\]
Now, for every $n\in\mbb{N},$\vspace*{-2pt}
%
\begin{eqnarray}\label{bor2pr}
\int_{0}^{1}\zeta^{\alpha,\beta}_{n}{R}_{n}^{\alpha,\beta
}(xz)D_{\alpha
,\mu}(\mathrm{d}z)&=&\int_{0}^{1}{Q}_{n}^{\alpha,\beta}(xz,1)D_{\alpha,\mu
}(\mathrm{d}z) \nonumber\\[-3pt]
&=&\sum_{m\leq
n}a_{nm}\frac{\ifact{(\alpha+\beta)}{m}}{\ifact{(\alpha)}{m}}\int
_0^1(xz)^mD_{\alpha,\mu}(\mathrm{d}z)
\nonumber
\\[-9.5pt]
\\[-9.5pt]
\nonumber
&=& \sum_{m\leq
n}a_{nm}\frac{\ifact{(\alpha+\beta)}{m}}{\ifact{(\alpha
)}{m}}\frac
{\ifact{(\alpha)}{m}}{\ifact{(\alpha+\mu)}{m}}x^m\\[-3pt]
&=&\zeta_n^{\alpha
+\mu
,\beta-\mu}{R}_{n}^{a+\mu,b-\mu}(x),\nonumber\vspace*{-2pt}
\end{eqnarray}
and this proves (\ref{01borr2}).

To prove (\ref{01borr1}), simply remember (see, e.g.,~\cite{GS08},
Section~3.1) that\vspace*{-2pt}
\[
R_{n}^{\alpha,\beta}(0)=(-1)^{n}\frac{\ifact{\alpha}{n}}{\ifact
{\beta}{n}}\vspace*{-2pt}
\]
and that\vspace*{-2pt}
\[
R_n^{\alpha,\beta}(x)=\frac{R_n^{\beta,\alpha}(1-x)}{R_{n}^{\beta
,\alpha}(0)}.\vspace*{-2pt}
\]
So we can use (\ref{01borr2}) to see that\vspace*{-2pt}
%
%
\begin{eqnarray}
\int_0^1\frac{R_n^{\beta,\alpha}((1-x)z)}{R_n^{\beta,\alpha}(0)}D_{\beta
,\mu}(\mathrm{d}z)&=&(-1)^n\frac{\ifact{\alpha}{n}}{\ifact{\beta}{n}}
\frac{\zeta_n^{\beta+\mu,\alpha-\mu}}{\zeta_n^{\beta,\alpha}}R_n^{\beta
+\mu,\alpha-\mu}(1-x)\nonumber
\\[-9.5pt]
\\[-9.5pt]
\nonumber
&=&\zeta_n^{\alpha-\mu,\beta+\mu}(x),\vspace*{-2pt}
\end{eqnarray}
and the proof is complete.\vadjust{\goodbreak}
\end{pf}

Lemma~\ref{l:bor} completes Theorem~\ref{th:2}:

\begin{corollary}
\label{cor:bpd} Let $\alpha\leq\beta$ with $\alpha+\beta\geq2.$
If a
sequence $\rho_{n}$ is positive-definite for $D_{\alpha,\beta},$ then
it is
positive-definite for $D_{\alpha+\mu,\beta-\mu},$ for any $0\leq
\mu\leq
\beta$.
\end{corollary}

\begin{pf}
By Theorem~\ref{th:2} $\rho_{n}$ is positive-definite for
$D_{\alpha,\beta}$ if and only if
\[
\sum_n\zeta_n^{\alpha,\beta}\rho_nR_n^{\alpha,\beta}(x)\geq0.
\]
So (\ref{01borr2}) implies also
\[
\sum_n\zeta_n^{\alpha,\beta}\rho_n\frac{\zeta_n^{\alpha+\mu
,\beta-\mu
}}{\zeta_{n}^{\alpha,\beta}}R_n^{\alpha+\mu,\beta-\mu}(x)\geq0.
\]
The case for $D_{\alpha-\mu,\beta+\mu}$ is proved similarly, but
using (\ref{01borr1}) instead of (\ref{01borr2}).
\end{pf}

For $d>2$, Proposition~\ref{PROP:1} leads to a similar
characterization
of all positive-definite sequences, for the Dirichlet distribution,
which are indexed only by their total degree, that is, all sequences
$\rho_{\mb{n}}=\rho_{|\mb{n}|}$.

\begin{proposition}
\label{pdprop} Let $\alpha\in{R}^{d}$ satisfy the same conditions as in
Proposition~\ref{PROP:1}. A sequence
$\{\rho_{n}=\rho_{ n}\dvt n\in\mathbb{N}\}$ is positive-definite for the
$\operatorname{Dirichlet} (\alpha)$ distribution if and only if it is positive-definite
for $D_{c|\alpha|, (1-c)|\alpha|}$, for every $c\in(0,1).$
\end{proposition}

\begin{pf}
\textit{Sufficiency}. First notice that, since
%
%
\begin{equation}
Q_{n}^{\alpha,\beta}(x,y)=Q_{n}^{\beta,\alpha}(1-x,1-y), \label{invert}
\end{equation}
then a sequence is positive-definite for $D_{\alpha,\beta}$ if and only
if it is positive definite for $D_{\beta,\alpha}$, so that we can
assume, without loss of generality, that $c|\alpha|\leq(1-c)|\alpha|.$
Let $\alpha=(\alpha_1\geq\alpha_2\geq\cdots\geq\alpha_d)$
satisfy the
conditions of Proposition~\ref{PROP:1} (again, the decreasing order is
assumed for simplicity) and let
\[
\sum_{ n=0}^{\infty}\rho_{ n}Q_{ n}^{c|\alpha|,(1-c)|\alpha
|}(u,v)\geq
0,\qquad u,v\in[0,1].
\]
If $\alpha_{d}>c_{|\alpha|}$
then Corollary~\ref{cor:bpd}, applied with
$\mu=\alpha_{d}-c_{|\alpha|}$ implies that
\[
\sum_{ n=0}^{\infty}\rho_{ n}Q_{ n}^{\alpha_{d},|\alpha|-\alpha
_{d}}(u,v)\geq0
\]
so by Proposition~\ref{PROP:1}
%
%
\begin{equation}
0\leq\int
\Biggl[\sum_{ n=0}^{\infty}\rho_{ n}Q_{ n}^{\alpha_{d},|\alpha|-\alpha
_{d}}(z_{d},1)\Biggr]m_{\mb{x},\mb{y};\alpha}(\mathrm{d}z_d)=\sum_{ n=0}^{\infty
}\rho_{ n}Q_{ n}^{\alpha}(\mb{x},\mb{y}),\qquad x,y\in\Delta_{(d-1)}.
\label{dpd1}
\end{equation}
If $\alpha_{d}<c_{|\alpha|},$ then apply Corollary~\ref{cor:bpd} with
$\mu
=|\alpha|(1-c)-\alpha_d$ to obtain
\[
\sum_{ n=0}^{\infty}\rho_{ n}Q_{ n}^{|\alpha|-\alpha_{d},\alpha
_{d}}(u,v)=\sum_{ n=0}^{\infty}\rho_{ n}Q_{ n}^{\alpha_{d},|\alpha
|-\alpha_{d}}(1-u,1-v)\geq0,
\]
which implies again (\ref{dpd1}), thus $\{\rho_{ n}\}$ is
positive-definite for $D_{\alpha}$.

\textit{Necessity.} For $I\subset\{1,\ldots,d\}$, the
random variables
\[
X_{I}=\sum_{j\in I}X_{j}; \qquad Y_{I}=\sum_{j\in I}Y_{j}
\]
have a $\operatorname{Beta}(\alpha_{I},|\alpha|-\alpha_{I})$ distribution, where
$\alpha_{I}=\sum_{j\in I}\alpha_{j}$. Since
\[
\mbb{E}\bigl(Q^{\alpha}_{n}(\mb{X},\mb{Y})|Y_{I}=z\bigr)=Q_{n}^{\alpha_{I}}(z),
\]
then for arbitrary $x,y\in\Delta_{(n-1)},$
\[
\sum_{ n=0}^{\infty}\rho_{ n}Q_{ n}^{\alpha}(\mb{x},\mb{y})\geq0
\]
implies
\[
Q_{ n}^{\alpha_{I},|\alpha|-\alpha_{I}}(x,y)=Q_{ n}^{|\alpha
|-\alpha
_{I},\alpha_{I}}(1-x,1-y)\geq0.
\]
Now we can apply once again
Corollary~\ref{cor:bpd} with $\mu=\pm(c|\alpha|-\alpha_{I})$
(whichever is
positive) to obtain, with the possible help of (\ref{invert}),
\[
\sum_{ n=0}^{\infty}\rho_{ n}Q_{ n}^{c|\alpha|,(1-c)|\alpha
|}(u,v)\geq
0,\qquad u,v\in[0,1].
\]
\upqed\end{pf}

\section{A probabilistic derivation of Jacobi positive-definite
sequences}\label{sec:prob}
In the previous sections we have found characterizations of Dirichlet
positive-definite sequences holding only if the parameters satisfied a
particular set of constraints. Here we show some sufficient conditions
for a sequence to be $\alpha$-JPDS, not requiring any constraints on
$\alpha.$ Thus we will identify a convex set of Jacobi
positive-definite sequences satisfying the property (P1), as shown in
Proposition~\ref{pr:gspd}. This is done by exploiting the probabilistic
interpretation of the orthogonal polynomial kernels. Let us reconsider
the function $\xi_{ m}^{\alpha}.$
The bivariate measure
%
%
\begin{equation}
\mathcal{BD}_{\alpha, m}(\mathrm{d}\mb{x},\mathrm{d}\mb{y}):=\xi^{\alpha}_{ m}(\mb
{x},\mb
{y})D_{\alpha}(\mathrm{d}\mb{x})D_{\alpha}(\mathrm{d}\mb{y}),
\label{bdm}
\end{equation}
so $\xi^{\alpha}_{ m}(\mb{x},\mb{y})$ has the interpretation as a
(exchangeable) copula for the joint law of two vectors $(X,Y),$ with identical
Dirichlet marginal distribution, arising as the\vadjust{\goodbreak} limit distribution of
colors in two P\'{o}lya urns with $ m$ random draws in common. Such a
joint law can be simulated via the following Gibbs sampling scheme:

\begin{longlist}
\item[(i)] \textit{Generate a vector $\mb{X}$ of $\operatorname{Dirichlet}(\alpha)$
random
frequencies on $d$ points.}

\item[(ii)] \textit{Conditional on the observed $\mb{X}=\mb{x}$, sample $ m$ i.i.d.
observations with common law $\mb{x}$.}

\item[(iii)] \textit{Given the vector $\mb{l}\in m\Delta_{(d-1)}$, which counts how
many observations in the sample at step \textup{(ii)} are equal to $1, \ldots
,d$, take $\mb{Y}$ as conditionally independent of $\mb{X}$ and
with distribution $D_{\alpha+\mb{l}}(\mathrm{d}\mb{y})$.}
\end{longlist}

The bivariate measure $\mathcal{BD}_{\alpha, m} $ and its
infinite-dimensional extension
has found several applications in Bayesian statistics (e.g., by \cite
{WSM05}), but no connections were made with orthogonal kernel and
canonical correlation sequences. A
recent important development of this direction is in~\cite{DKSC08}.

Now, let us allow the number $m$ of random draws in common to
be a random number $M$, say, assume that the probability that the two
P\'{o}lya urns have $M=m$ draws in common is $d_m,$ for any probability
distribution
$\{d_{ m}\dvt m=0,1,2,\ldots\}$ on $\mbb{N}.$
Then we obtain a new joint distribution, with identical
Dirichlet marginals and
copula given by
%
%
\begin{equation} \label{xicop}\mathcal{BD}_{\alpha,d}(\mathrm{d}\mb{x},\mathrm{d}\mb
{y})=\mbb{E}[\mathcal{BD}_{\alpha, M}(\mathrm{d}\mb{x},\mathrm{d}\mb{y})]=\sum
_{ m=0}^{\infty}d_{ m}\xi^{\alpha}_{ m}(\mb{x},\mb{y})D_{\alpha
}(\mathrm{d}\mb
{x})D_{\alpha}(\mathrm{d}\mb{y}).
\end{equation}
The probabilistic
construction has just led us to prove the following:
\begin{proposition}
\label{pr:gspd} Let $\{d_{ m}\dvt m=0,1,\ldots\}$ be a probability measure
on $\{0,1,2,\ldots\}.$ For every $|\theta|\geq0,$ the sequence
%
%
\begin{equation}
\rho_{ n}=\sum_{ m\geq n}\frac{ m_{[ n]}}{\ifact{(|\theta|+ m)}{ n}}d_{
m},  \qquad n=0,1,2,\ldots,\label{gspd}
\end{equation}
is $\alpha$-JPDS \textit{for every $d$ and every
$\alpha\in\mbb{R}^{d}$ such that $|\alpha|=|\theta|.$}
\end{proposition}

\begin{pf}
Note that
\[
\rho_{0}=\sum_{ m=0}^{\infty}d_{ m}=1
\]
is always true for every probability measure $\{d_{ m}\}.$

Now reconsider the form
(\ref{inverse}) for the (positive) function $\xi^{\alpha}_{ m}$: we can
rewrite (\ref{xicop}) as
%
%
\begin{eqnarray}\label{probqpd}
0&\leq&\sum_{ m=0}^{\infty}d_{ m}\xi^{\alpha}_{ m}(\mb{x},\mb{y})
\nonumber\\
&=&\sum_{ m=0}^{\infty}d_{ m}\sum_{ n\leq m}\frac{ m_{[ n]}}{\ifact
{(|\theta|+ m)}{ n}}Q_{ n}^{\alpha}(\mb{x},\mb{y})
\nonumber
\\[-8pt]
\\[-8pt]
\nonumber
&=&\sum_{ n=0}^{\infty}\biggl[\sum_{ m\geq n}\frac{ m_{[ n]}}
{\ifact{(|\theta|+ m)}{ n}}d_{ m}\biggr]Q^{\alpha}_{ m}(\mb{x},\mb
{y}) \\
&=&\sum_{ n=0}^{\infty}\rho_{ n}Q^{\alpha}_{ m}(\mb{x},\mb
{y}),\nonumber
\end{eqnarray}
and since (\ref{probqpd}) does not depend on the dimension of
$\alpha,$ then the proposition is proved for $D_{\alpha}.$
\end{pf}

\begin{example}
Take $d_{ m}=\delta_{ m l}$, the probability assigning full mass to $
l.$ The corresponding positive-definite sequence is
%
%
\begin{equation}\rho_{ n}(m)=\sum_{ m\geq n}\frac{ m_{[ n]}}{\ifact
{(|\theta|+ m)}{ n}}\delta_{ m l}=\frac{ l_{[ n]}}{\ifact{(|\theta|+
l)}{ n}}\mbb{I}( l\geq n).
\label{diracpd}
\end{equation}
and by Proposition~\ref{pr:qn},
%
%
\begin{equation}
\sum_{ n=0}^{\infty}\rho_{ n}(m)Q_{ n}^{\alpha}(\mb{x},\mb{y})=
\sum_{
n=0}^{ l}\frac{ l_{[ n]}}{\ifact{(|\theta|+ l)}{ n}}Q_{ n}^{\alpha
}(\mb
{x},\mb{y})=\xi^{\alpha}_{ l}(\mb{x},\mb{y})\geq0. \label{xidirac}
\end{equation}
Thus $\rho_n(m)$ forms the sequences of canonical correlations induced
by the bivariate probability distribution $\xi^{\alpha}_{ l}(\mb
{x},\mb
{y})D_\alpha(\mb{x})D_\alpha(\mb{y}).$
\end{example}

\begin{example}\label{ex:gen}
Consider, for every $t\geq0,$ the probability distribution
%
%
\begin{equation}
d_{ m}(t)=\sum_{ n\geq n}a^{|\alpha|}_{ m n}\mathrm{e}^{-({1}/{2}) n(
n+|\alpha|-1)t},\qquad   m=0,1,2\ldots,
\label{genex}
\end{equation}
where $(a^{|\alpha|}_{ m n})$ is the invertible triangular system
(\ref
{anm}) defining the polynomial kernels $Q_{ n}^{\alpha}$ in Proposition
\ref{pr:qn}.
Since the coefficients of the inverse system are exactly of the form
\[
\frac{ m_{[ n]}}{\ifact{(|\theta|+ m)}{ n}}, \qquad  m, n=0,1,2,\ldots,
\]
then
\[
\rho_{ n}(t)=\mathrm{e}^{-({1}/{2}) n( n+|\alpha|-1)t}
\]
is, for every $t$, a positive-definite sequence. In particular, it is
the one characterizing the neutral Wright--Fisher diffusion in population
genetics, mentioned in Section~\ref{sec1.1}, whose generator has eigenvalues
$-\frac{1}{2} n( n+|\alpha|-1)$ and orthogonal polynomial
eigenfunctions.

The distribution (\ref{genex}) is the so-called \textit{coalescent
lineage distribution} (see~\cite{G81,G06}), that is, the
probability distribution of the number of lineages surviving up to
time $t$ back in the past, when the total mutation rate is $|\alpha|$, and
the allele frequencies of $d$ phenotypes in the whole population are
governed by $A_{|\alpha|,d}.$ More details on the connection between
coalescent lineage distributions and Jacobi polynomials can be found in
\cite{GS10}.
\end{example}
\begin{example}[(Perfect independence and dependence)]\label{ex:pf}
Extreme cases of perfect dependence or perfect independence can be
obtained from Example~\ref{ex:gen}, when we take the limit as
$t\rightarrow0$ or $t\rightarrow\infty$, respectively. In the former
case, $d_{ m}(0)=\delta_{ m\infty}$ so that $\rho_{ n}(0)=1$ for every
$ n$. The corresponding bivariate distribution is such that
\[
\mbb{E}_0\bigl(Q_{\mb{n}}(\mb{Y})|\mb{X}=\mb{x}\bigr)=Q_{\mb{n}}(\mb{x})
\]
so that, for every square-integrable function
\[
f=\sum_{\mb{n}}c_{\mb{n}}Q_{\mb{n}},
\]
we have
\[
\mbb{E}_0\bigl(f(\mb{Y})|\mb{X}=\mb{x}\bigr)=\sum_{\mb{n}}c_{\mb{n}}Q_{\mb
{n}}(\mb
{x})=f(\mb{x});
\]
that is, $\mathcal{BD}_{\alpha,\{0\}}$ is, in fact, the Dirac measure
$\delta(\mb{y}-\mb{x}).$

In the latter case, $d_{ m}(\infty)=\delta_{ m0}$ so that $\rho_{
n}(\infty)=0$ for every $ n>1$ and $\mbb{E}_0(Q_n(\mb{Y})|\mb
{X}=\mb
{x})=\mbb{E}[Q_n(\mb{Y})],$
implying that
\[
\mbb{E}_\infty\bigl(f(\mb{Y})|\mb{X}=\mb{x}\bigr)=\mbb{E}[f(\mb{Y})],
\]
that is, $\mb{X},\mb{Y}$ are stochastically independent.
\end{example}

\subsection{The infinite-dimensional case}
Proposition~\ref{pr:gspd} also extends to Poisson--Dirichlet measures.
The argument and construction are the same, once one replaces $\xi_{
m}^{\alpha}$ with $\xi_{ m}^{\downarrow|\theta|,\infty}.$ We only need
to observe that because
the functions
\[
\pmatrix{{ m}\cr{\mb{l}}}\sharp(\mb{l})[\mb{x},\mb{l}],
\]
forming the terms in $\xi_{ m}^{\downarrow|\theta|,\infty}$ (see
(\ref
{symxi})), are probability measures on $ m\Delta_{\infty}^{\downarrow
},$ then the kernel
\[
\xi_{ m}^{\downarrow|\theta|,\infty}(\mb{x},\mb{y})D^{\downarrow
}_{|\theta|,\infty}(\mathrm{d}\mb{y})
\]
defines, for every $\mb{x},$ a proper transition probability function
on $\Delta^{\downarrow}_{\infty},$ allowing for the Gibbs sampling
interpretation as in Section~\ref{sec:prob}, but are modified as follows:

\begin{longlist}
\item[(i)] \textit{Generate a point $\mb{X}$ in $\Delta
^{\downarrow
}_{\infty}$ with distribution $\operatorname{PD}(|\theta|)$.}

\item[(ii)] \textit{Conditional on the observed $\mb{X}=\mb{x}$, sample a partition of
$ m$ with distribution function ${{ m}\choose{\mb{l}}}\sharp(\mb
{l})[\mb{x},\mb{l}]$.}

\item[(iii)] \textit{Conditionally on the vector $\mb{l}$, counting the cardinalities
of the blocks in the partition obtained at step \textup{(ii)}, take $Y$ as
stochastically independent of $X$ and
with distribution}
\[
\frac{{{ m}\choose{\mb{l}}}\sharp(\mb{l})[\mb{x},\mb
{l}]\operatorname{PD}_{\theta
}(\mathrm{d}\mb{y})}{\operatorname{ESF}_{|\theta|}(\mb{l})}.
\]
\end{longlist}

Thus the proof of the following statement is now obvious.
\begin{proposition}
Let $\{d_{ m}\dvt m=0,1,\ldots\}$ be a probability measure on $\{
0,1,2,\ldots\}.$ For every $|\theta|\geq0,$ the sequence
%
%
\begin{equation}
\rho_{ n}=\sum_{ m\geq n}\frac{ m_{[ n]}}{\ifact{(|\theta|+ m)}{ n}}d_{
m}, \qquad n=0,1,2,\ldots,\label{gspdinfty}
\end{equation}
is a positive-definite sequence for the Poisson--Dirichlet point
process with parameter $|\theta|.$
\end{proposition}
%

\section{From positive-definite sequences to probability
measures}\label
{sec:pd2prob}
In the previous section we have seen that it is possible to map
probability distributions on $\mbb{Z}_+$ to Jacobi positive-definite
sequences. It is natural to ask if, on the other way around, JPDSs $\{
\rho_{ n}\}$ can be mapped to probability distributions $\{d_{ m}\}$ on
$\mbb{Z}_+,$ for
every $m=0,1,\ldots,$ via the inversion\vspace*{-2pt}
%
%
\begin{equation}d_{ m}(\rho)=\sum_{n=m}^{\infty}a^{|\alpha|}_{ n
m}\rho
_{ n}\label{dmrho}.\vspace*{-2pt}
\end{equation}
For this to happen we only need $d_m(\rho)$ to be non-negative for
every $m$ as it is easy to check that $\sum_md_m(\rho)=1$ always. In
this section we give some sufficient conditions on $\rho$ for $d_{
m}(\rho)$ to be non-negative for every $ m=0,1,\ldots,$ and an
important counterexample showing that not all JPDSs can be associated
to probabilities.
We restrict our attention to the Beta case ($d=2$) as we now know that,
if associated to a probability on $\mbb{Z}_+$, any JPDS for $d=2$ is
also JPDS for $d>2$.

Suppose $\rho=\{\rho_{ n}\}_{ n=0}^{\infty}$ satisfies\vspace*{-2pt}
%
%
\begin{equation}
p_{\rho}(x,y):=\sum_{ n=0}^{\infty}\rho_{ n}Q_{ n}^{\alpha,\beta
}(x,y)\geq
0\label{bk}\vspace*{-2pt}
\end{equation}
and, in
particular,\vspace*{-2pt}
%
%
\begin{equation}
p_{\rho}(x):=p_{\rho}(x,1)\geq0\label{pdro}.\vspace*{-2pt}
\end{equation}

\begin{proposition}
If all the derivatives of $p_{\rho}(x)$ exist, then $d_{ m}(\rho)\geq
0$ for every $ m\in\mbb{Z}_+$ if and only if all derivatives of
$p_{\rho}(x)$ are
non-negative.
\end{proposition}

\begin{pf}
Rewrite $d_{ m}(\rho)$ as\vspace*{-2pt}
%
%
\begin{eqnarray}\label{shift}
d_{ m}(\rho)&=&\sum_{ v=0}^{\infty}a_{ v+ m, m}^{|\theta|}\rho_{ v+
m}
\nonumber
\\[-8pt]
\\[-8pt]
\nonumber
&=&\frac{(|\theta|+ m)_{( m)}}{ m!}\sum_{ v=0}^{\infty}a_{
v0}^{|\theta
|+2 m}\rho_{ v+ m},\qquad  m=0,1,\ldots.\vspace*{-2pt}
\end{eqnarray}
This follows from the general identity\vspace*{-2pt}
%
%
\begin{equation}
a^{|\theta|}_{ v+ j, u+ j}=a_{ v, u}^{|\theta|+2 j}\frac{ u!}{( u+
j)!}(|\theta|+ u+ j)_{( j)}.
\label{shiftavu}
\end{equation}
Now consider the expansion of Jacobi polynomials. We know that
%
%
\begin{eqnarray}\label{qxi}
\zeta_{ n}^{\alpha,\beta}R_{ n}^{\alpha,\beta}(x)R_{ n}^{\alpha
,\beta
}(y)&=&Q_{ n}^{\alpha,\beta}(x,y)
\nonumber
\\[-8pt]
\\[-8pt]
\nonumber
&=&\sum_{ m=0}^{n}a_{ n m}^{|\theta|}\xi_{ m}^{\alpha,\beta
}(x,y).
\end{eqnarray}
%
Since $R_{ n}^{\alpha,\beta}(1)=1$ and
$\xi_{ m}^{\alpha,\beta}(0,1)=\delta_{ m0}$, then
%
%
\begin{equation}
\zeta_{ n}^{\alpha,\beta}R_{ n}^{\alpha,\beta}(0)=Q_{ n}^{\alpha
,\beta
}(0,1)=a_{ n0}^{|\theta|}\label{an0}.
\end{equation}
Therefore (\ref{shift}) becomes
%
%
\begin{equation}
d_{ m}(\rho)=\frac{(|\theta|+ m)_{( m)}}{ m!}\sum_{ v=0}^{\infty
}\zeta
_v^{\alpha+ m,\beta+ m}R_{ v}^{\alpha+ m,\beta+ m}(0)\rho_{ v+
m},\qquad m=0,1,\ldots.
\label{shift2}
\end{equation}
Now apply, for example,~\cite{I05}, (4.3.2), to deduce
%
%
\begin{eqnarray}\label{isid}
\frac{\mathrm{d}^{ m}}{\mathrm{d}y^{ m}}[D_{\alpha+ m,\beta+ m}(y)R_{ v}^{\alpha+
m,\beta+ m}(y)]=
(-1)^{ m}\frac{\theta_{(2 m)}}{\alpha_{( m)}}R_{ v+ m}^{\alpha
,\beta
}(y)D_{\alpha,\beta}(y).
\end{eqnarray}
For $ m=1$,
\begin{eqnarray*}\label{inpart}
\rho_{ v+1}&=&\int_0^1p_{\rho}(x)R_{ v+1}^{\alpha,\beta
}(x)D_{\alpha
,\beta}(x)\,\mathrm{d}x \\
&=&-\frac{\alpha}{|\theta|_{(2)}}\int_0^1p_{\rho}(x)\biggl[\frac
{\mathrm{d}}{\mathrm{d}x}R_{ v}^{\alpha+1,\beta+1}(x)D_{\alpha+1,\beta+1}(x)
\biggr]\,\mathrm{d}x \\
&=&\frac{\alpha}{|\theta|_{(2)}}\int_0^1\biggl(\frac{\mathrm{d}}{\mathrm{d}x}p_{\rho
}(x)\biggr)R_{ v}^{\alpha+1,\beta+1}(x)D_{\alpha+1,\beta+1}(x)\,\mathrm{d}x
.
\end{eqnarray*}
The last equality is obtained after integrating by parts. Similarly, denote
\[
p^{( m)}_{\rho}(x):=\frac{\mathrm{d}^{ m}}{\mathrm{d}x^{ m}}p_{\rho}(x),
\qquad m=0,1,\ldots.
\]
It is easy to prove
that
%
%
\begin{equation}\rho_{ v+ m}=\frac{ m!\alpha_{( m)}}{|\theta|_{(2 m)}}
\int_0^1p^{( m)}_{\rho}(x)R_{ v}^{\alpha+ m,\beta+ m}(x)D_{\alpha+
m,\beta+ m}(x)\,\mathrm{d}x
\label{diffid} ,\vspace*{-1pt}
\end{equation}
so we can write
\[
d_{ m}(\rho)=\frac{\alpha_{( m)}}{|\theta|_{( m)}}p^{( m)}_{\rho}(0).\vspace*{-1pt}
\]
Thus if $p^{( m)}_{\rho}\geq0,$ then $d_{ m}(\rho)$ is,
for every $ m$, non-negative and this proves the sufficiency.

For the necessity, assume, without loss of generality,
that $\{d_{ m}(\rho)\dvt m\in\mbb{Z}_+\}$ is a probability mass function
on $\mbb{Z}_+$. Then its probability generating function (\textit{p.g.f.})
must have
all derivatives non-negative. For every $0<\gamma<|\theta|,$ the p.g.f.
has the representation
%
%
\begin{eqnarray}\label{pgf}
\varphi(s)&=&\sum_{ m=0}^{\infty}d_{ m}(\rho)s^{ m} \nonumber\\[-2pt]
&=&\mbb{E}_{\gamma,|\theta|-\gamma}\Biggl[\sum_{ m=0}^{\infty}d_{
m}(\rho
)\xi_{ m}^{\gamma,|\theta|-\gamma}(sZ,1)\Biggr]
\nonumber
\\[-9pt]
\\[-9pt]
\nonumber
&=&\mbb{E}_{\gamma,|\theta|-\gamma}\Biggl[\sum_{ m=0}^{\infty}\rho_{ n}
\zeta_{ n}^{\gamma,|\theta|-\gamma}R_{ n}^{\gamma,|\theta|-\gamma
}(sZ)\Biggr] \\[-2pt]
&=&\mbb{E}_{\gamma,|\theta|-\gamma}[p_{\rho}(sZ)],\nonumber
\end{eqnarray}
where $Z$ is a $\operatorname{Beta}(\gamma,|\theta|-\gamma)$ random variable. Here
the second equality follows from the identity
%
%
\begin{equation}
\frac{|\theta|_{( m)}}{\alpha_{( m)}}x^{ m}=\xi_{ m}^{\alpha,\beta
}(x,1), \qquad \alpha,\beta>0 \label{ximom},\vspace*{-1pt}
\end{equation}
and the third equality
comes from (\ref{qxi}).

So, for every $k=0,1,\ldots,$
%
%
\begin{equation}
0\leq
\frac{\mathrm{d}^{ k}}{\mathrm{d}s^{ k}}\varphi(s)=\mbb{E}_{\gamma,|\theta|-\gamma}
\bigl[Z^{ k}
p^{( k)}_{\rho}(sZ)\bigr]\label{pgfder}\vspace*{-1pt}
\end{equation}
for every
$\gamma\in(0,|\theta|).$ Now, if we take the limit as
$\gamma\rightarrow|\theta|$, $Z\mathop{\rightarrow}^{d}1$ so, by
continuity,
\[
\mbb{E}_{\gamma,|\theta|-\gamma}\bigl[Z^{ k}
p^{( k)}_{\rho}(sZ)\bigr]\mathop{\rightarrow}_{\gamma\rightarrow|\theta|}
p^{({ k})}(s),\vspace*{-1pt}
\]
preserving the positivity, which completes the proof.\vspace*{-2pt}
\end{pf}
%
\subsection{A counterexample}\vspace*{-2pt}
In Gasper's representation (Theorem~\ref{th:2}), every
positive-definite sequence is a mixture of Jacobi polynomials,
normalized with respect to their value at 1. It is natural to ask
whether these extreme\vadjust{\goodbreak} points lend themselves to probability
measures on $\mbb{Z}_+.$ A positive answer would imply that all
positive-definite sequences, under Gasper's conditions, are
coupled with probabilities on the integers. Rather surprisingly,
the answer is negative.

\begin{proposition}\label{prob:2}Let $\alpha,\beta>0$ satisfy Gasper's
conditions. The function
\[
d_{ m}=\sum_{ n\geq|m|}a^{|\theta|}_{ n m}R_{ n}^{\alpha,\beta
}(x), \qquad m=0,1,2,\ldots,
\]
is not a probability
measure.
\end{proposition}
\begin{pf}
Rewrite
%
%
\begin{eqnarray}\label{gf2} \phi_x(s)&=&\sum_{n=0}^{\infty}R_{n}^{\alpha,\beta
}(x)\sum_{ m=0}^{ n}a^{|\theta|}_{ n m} s^{ m}
\nonumber
\\[-8pt]
\\[-8pt]
\nonumber
&=&\mbb{E}\sum_{ n=0}^{\infty}\zeta_{ n}^{\alpha,\beta}R_{
n}^{\alpha
,\beta}(x)R_{ n}^{\alpha,\beta}(Ws),
\end{eqnarray}
where $W$ is a $\operatorname{Beta} (\alpha,\beta)$ random variable.
This also shows that, for every $x,$
\[
\frac{\mathrm{d}D_{\alpha,\beta}(y)}{\mathrm{d}y}\sum_{n=0}^{\infty}\zeta_{
n}^{\alpha
,\beta}R_{ n}^{\alpha,\beta}(x)R_{ n}^{\alpha,\beta}(y)=\delta_x(y),
\]
that is, the Dirac measure putting all its unit mass on $x$ (see also
Example~\ref{ex:pf}).

Now, if $\phi_x(s)$ is a probability generating function, then,
for every positive $L_2$ function $g,$
any mixture of the form
%
%
\begin{eqnarray}\label{mixpgf}
q(s)&=&\int_0^1g(x)\phi_x(s)\frac{x^{\alpha-1}(1-x)^{\beta
-1}}{B(\alpha
,\beta)}\,\mathrm{d}x
\nonumber
\\[-8pt]
\\[-8pt]
\nonumber
&=&\int_0^1g(ws)\frac{w^{\alpha-1}(1-w)^{\beta-1}}{B(\alpha,\beta
)}\,\mathrm{d}w
\end{eqnarray}
must be a probability generating function; that is, it must have all
derivatives positive. However, if we choose $g(x)=\mathrm{e}^{-\lambda x},$
then we know that, $g$ being completely monotone, the derivatives
of $q$ will have alternating sign, which proves the
claim.
\end{pf}
%
\section{Positive-definite sequences in the
Dirichlet multinomial distribution}\label{sec:hpds}
In this section we aim to investigate the relationship
existing between JPDS and HPDS. In particular, we wish to understand
when (P2) is true, that is, when a sequence is both HPDS and JPDS for a
given $\alpha$. It turns out that, by using the results in Sections
\ref
{sec:k2} and~\ref{sec:irhpk}, it is possible to define several
(sometimes striking) mappings from JPDS and HPDS and vice versa, but we
could prove~(P2) only for particular subclasses of positive-definite
sequences. In Proposition~\ref{bernpds} we prove that every JPDS is a
limit of~(P2) sequences. Later, in Proposition~\ref{prp:p2}, we will
identify another~(P2) family of positive-definite sequences, as a
proper subfamily of the JPDSs, derived in Section~\ref{sec:prob} as the
image, under a specific bijection, of a probability on $\mbb{Z}_+$.

The first proposition holds with no constraints on $\alpha$ or $d.$
\begin{proposition}\label{pr:j2h1}
For every $d$ and $\alpha\in\mbb{R}^d_+$, let $\rho=\{\rho_{ n}\}$ be
a $\alpha$-JPDS. Then
%
%
\begin{equation}\rho_{ n}\frac{N_{[ n]}}{\ifact{(|\alpha|+N)}{
n}},\qquad n=0,1,2,\ldots,\label{hpd}
\end{equation}
is a positive-definite sequence for $\operatorname{DM}_{\alpha,N}, $ for
every $N=1,2,\ldots.$

\end{proposition}
\begin{pf}
From Proposition~\ref{pr:KDM}, if
\[
\sum_{ n=0}^{\infty}\rho_{ n}Q_{ n}^{\alpha}(\mb{x},\mb{y})\geq0,
\]
then for every $\mb{r},\mb{s}\in\mbb{N}^{d}\dvt |\mb{r}|=|\mb{s}|=N$,
\[
\sum_{ n=0}^{\infty}\rho_{ n}\int\int
Q_{ n}^{\alpha}(\mb{x},\mb{y})D_{\alpha+\mb{r}}(\mathrm{d}\mb{x})D_{\alpha
+\mb
{s}}(\mathrm{d}\mb{y})=\sum_{ n=0}^{\infty}\rho_{ n}\frac{N_{[ n]}}{\ifact
{(|\alpha|+N)}{ n}}H^{\alpha}_{ n}(\mb{r},\mb{s})\geq
0.
\]
\upqed\end{pf}
\begin{example}\label{ex:hgen}Consider the JPDS given in Example \ref
{ex:gen} from population genetics; $\rho_n(t)=\mathrm{e}^{-({1}/{2})tn(n+|\alpha|-1)}, t\geq0.$ The HPDS
%
%
\begin{equation}
\rho_n(t|N)=\frac{N_{[ n]}}{\ifact{(|\alpha|+N)}{ n}}\mathrm{e}^{-
({1}/{2})tn(n+|\alpha|-1)}\label{cldh}
\end{equation}
describes the survival function of the number of non-mutant surviving
lineages at time $t$ in the past, in a coalescent process with neutral
mutation, starting with $N$ surviving lineages at time 0 (see \cite
{G06} for more details and references).
\end{example}

Two important HPDSs are given in the following lemma.
\begin{lemma}
\label{l:hpd}
For every $d,$ every $ m\leq N$ and every $\alpha\in\mbb{R}^d_+,$
both sequences
%
%
\begin{equation}
\biggl\{\frac{ m_{[ n]}}{\ifact{(|\alpha|+ m)}{ n}}\frac{\ifact{(|\alpha
|+N)}{ n}}{N_{[ n]}}\biggr\}_{ n\in\mbb{Z}_+}
\label{2fracpd}
\end{equation}
and
%
%
\begin{equation}
\biggl\{\frac{ m_{[ n]}}{\ifact{(|\alpha|+ m)}{ n}}\biggr\}_{ n\in\mbb{Z}_+}
\label{1fracpd}
\end{equation}
are $\alpha$-HPDSs for $\operatorname{DM}_{\alpha,N}$.\vadjust{\goodbreak}
\end{lemma}
\begin{pf}
From Proposition~\ref{vdmk}, by inverting (\ref{n22}) we know that, for
$ m=0,\ldots,N$
\[
0\leq\chi_{ m}^{H,\alpha}=\sum_{ n=0}^{ m}\frac{ m_{[ n]}}{\ifact
{(|\alpha|+m)}{ n}}\frac{\ifact{(|\alpha|+N)}{ n}}{ N_{[ n]}} H_{
n}^{\alpha},
\]
so
\[
\biggl\{\frac{ m_{[ n]}}{\ifact{(|\alpha|+ m)}{ n}}\frac{\ifact{(|\alpha
|+N)}{ n}}{ N_{[ n]}}\biggr\}
\]
is a HPDS.

Now let $\wt{\rho}_{ n}$ be a JPDS. By Proposition~\ref{pr:j2h1},
the sequence
\[
\biggl\{\wt{\rho}_{ n}\frac{ N_{[ n]}}{\ifact{(|\alpha|+N)}{ n}}\biggr\}
\]
is $\alpha$-HPDS. By multiplication,
\[
\biggl\{\wt{\rho}_{ n}\frac{N_{[ n]}}{\ifact{(|\alpha|+N)}{ n}}\frac{
m_{[ n]}}{\ifact{(|\alpha|+ m)}{ n}}\frac{\ifact{(|\alpha|+N)}{
n}}{N_{[ n]}}\biggr\}=
\biggl\{\wt{\rho}_{ n}\frac{ m_{[ n]}}{\ifact{(|\alpha|+ m)}{ n}}
\biggr\}
\]
is HPDS as well.
This also implies that
\[
\biggl\{\frac{ m_{[ n]}}{\ifact{(|\alpha|+ m)}{ n}}\biggr\}
\]
is HPDS (to convince oneself, take $(\wt{\rho}_n)$ as in Example \ref
{ex:gen} or in Example~\ref{ex:pf}, and take the limit as
$t\rightarrow
0$ or $z\rightarrow1$, respectively).
\end{pf}

We are now ready for our first result on (P2)-sequences.
\begin{proposition}
\label{bernpds}
For every $d$ and $\alpha\in\mbb{R}^d_+$, let $\rho=\{\rho_{ n}\}$ be
a $\alpha$-JPDS. Then there exists a sequence $\{\rho^{N}_{ n}\dvt n\in
\mbb
{Z}_+\}_{N=0}^{\infty}$, such that:
\begin{enumerate}[(ii)]
\item[(i)]for every $ n,$
\[
\rho_{ n}=\lim_{N\rightarrow\infty}\rho_{ n}^{N};
\]
\item[(ii)]
for every $N,$ the sequence $\{\rho_{ n}^{N}\}$ is both HPDS and JPDS.
\end{enumerate}
\end{proposition}

\begin{pf}
We show the proof for $d=2.$ For $d>2$ the proof is essentially the
same, with all distributions obviously replaced by their multivariate
versions. Take $I,J$ two independent $\operatorname{DM}_{(\alpha,\beta),N}$ and
$\operatorname{DM}_{(\alpha,\beta),M}$ random variables. As a result of de Finetti's
representation theorem, conditionally on the event $\{\lim
_{N\rightarrow
\infty}(\frac{I}{N}\frac{J}{M})=(x,y)\},$ the $(I,J)$ are independent
binomial r.v.s with parameter $(N,x)$ and $(M,y)$, respectively.

Let $f\dvtx[0,1]^2\rightarrow\mbb{R}$ be a positive continuous function.
The function
\[
B_{N,M}f(x,y):=\mbb{E}\biggl[f\biggl(\frac{I}{N},\frac{J}{M}\biggr) \big| x,y\biggr], \qquad
N,M=0,1,\ldots,
\]
is positive, as well, and, as $N,M\rightarrow\infty,$
\[
B_{N,M}f(x,y){\longrightarrow}f(x,y).
\]
Now take
\[
p_{\rho}(x,y)=\sum_{ n}\rho_{ n}Q^{\alpha,\beta}_{ n}(x,y)\geq0
\]
for every $x,y\in[0,1].$
Then, for $X,Y$ independent $D_{\alpha,\beta},$
\begin{eqnarray*}
\rho_{ n}&=&\mbb{E}[{Q^{\alpha,\beta}_{ n}(X,Y)}p_{\rho}(X,Y)
] \\
&=&\mbb{E}\Bigl[Q^{\alpha,\beta}_{ n}(X,Y)\lim_{N\rightarrow\infty
}B_{N,N}p_{\rho}(X,Y)\Bigr] \\
&=&\lim_{N\rightarrow\infty}\mbb{E}[Q^{\alpha,\beta}_{
n}(X,Y)B_{N,N}p_{\rho}(X,Y)] \\
&=&\lim_{N\rightarrow\infty}\rho_{ n}^{N},
\end{eqnarray*}
where
\[
\rho_{ n}^{N}:=\mbb{E}[Q^{\alpha,\beta}_{ n}(X,Y)B_{N,N}p_{\rho}
(X,Y)].
\]
But $B_{N,N}p_{\rho}$ is positive, so (i) is proved.

Now rewrite
%
%
\begin{eqnarray}\label{prehpds}
\rho_{ n}^{N}&=&\int_0^1\int_0^1\sum_{i=1}^{N}\sum
_{j=1}^{N}p_{\rho
}\biggl(\frac{i}{N},\frac{j}{N}\biggr)Q^{\alpha,\beta}_{
n}(x,y)\pmatrix{N\cr i}x^i(1-x)^{N-i}\nonumber\\
&&\hspace*{61pt}{}\times \pmatrix{N\cr j}y^j(1-y)^{N-j}D_{\alpha
}(\mathrm{d}x)D_{\alpha}(\mathrm{d}y)
\nonumber
\\[-8pt]
\\[-8pt]
\nonumber
&=&\sum_{i=1}^{N}\sum_{j=1}^{N}\operatorname{DM}_{\alpha,N}(i)\operatorname{DM}_{\alpha
,N}(j)p_{\rho
}\biggl(\frac{i}{N},\frac{j}{N}\biggr)\mbb{E}[Q^{\alpha,\beta}_{
n}(X,Y) |  i,j] \\
&=&\frac{N_{[n]}}{\ifact{(\alpha+\beta+N)}{ n}}\mbb{E}\biggl[p_{\rho
}\biggl(\frac{I}{N},\frac{J}{N}\biggr)H_{ n}^{\alpha,\beta}(I,J)\biggr]
\nonumber
\end{eqnarray}
for $I,J$ are independent $\operatorname{DM}_{(\alpha,\beta),N}$ random variables. The
last equality follows from (\ref{KDM}). Since~$p_{\rho}$ is positive,
from (\ref{prehpds}), it follows that
\[
\biggl\{\rho_{ n}^{N}\frac{\ifact{(\alpha+\beta+N)}{ n}}{N_{[n]}}\biggr\}
\]
is, for every $N,$ $\alpha$-HPDS. But by Lemma~\ref{l:hpd}, we can
multiply every term of the sequence by the HPDS (\ref{1fracpd}), where
we set $ m=N,$ to obtain (ii).
\end{pf}

The next proposition shows some mappings from Hahn to Jacobi
PDSs. It is, in some sense, a~converse of Proposition~\ref{pr:j2h1}
under the usual (extended) Gasper constraints on~$\alpha$.

\begin{proposition}
\label{jhpd}
If $\alpha$ satisfies the conditions of Proposition~\ref{PROP:1}, let
$\{\rho_{ n}\}$ be $\alpha$-HPDS for some integer $N$. Then both $\{
\rho
_{ n}\}$ and (\ref{hpd}) are positive definite for $D_{\alpha}$.
\end{proposition}

\begin{pf}
If
\[
\sum_{ n=0}^{\infty}\rho_{ n}H^{\alpha}_{ n}(\mb{r},\mb{s})\geq
0
\]
for every $\mb{r},\mb{s}\in N\Delta_{(d-1)},$ then Proposition \ref
{l:chu} implies that
\[
\sum_{ n=0}^{\infty}\rho_{ n}
{H}^{\alpha_1,|\alpha|-\alpha_1}_{ n}(r_1,N)\geq0.
\]
Now consider the Hahn polynomials re-normalized so that
\[
\wt{h}^{\alpha_1,|\alpha|-\alpha_1}_{ n}(r;N)=\int_0^1\frac
{Q^{\alpha
_1,|\alpha|-\alpha_1}_{ n}(x,1)}{Q^{\alpha_1,|\alpha|-\alpha_1}_{
n}(1,1)}D_{\alpha+{r}}(\mathrm{d}{x}).
\]
Then it is easy to prove that
\[
\wt{h}^{\alpha_1,|\alpha|-\alpha_1}_{ n}(N;N)=1
\]
and
\[
\mbb{E}\bigl[\wt{h}^{\alpha_1,|\alpha|-\alpha_1}_{ n}(R;N)\bigr]^2=
\frac{N_{[ n]}}{\ifact{(|\alpha|+N)}{ n}}\frac{1}{\zeta_{
n}^{\alpha
_1,|\alpha|-\alpha_1}},\qquad  n=0,1,\ldots
\]
(see also~\cite{GS08}, (5.65)).
Hence
%
%
\begin{eqnarray*}
0&\leq&\sum_{ n=0}^{\infty}\rho_{ n}{H}^{\alpha_1,|\alpha|-\alpha_1}_{
n}(r_1,N) \\
&=&\sum_{ n=0}^{\infty}\rho_{ n}\frac{\ifact{(|\alpha|+N)}{ n}}{N_{[
n]}}{\zeta_{ n}^{\alpha_1,|\alpha|-\alpha_1}}\wt{h}^{\alpha
_1,|\alpha
|-\alpha_1}_{ n}(r_1;N)=:f_N(r)
.\label{app1}
\end{eqnarray*}
So, for every $ n,$
%
%
\begin{eqnarray}
\rho_{ n}&=&\mbb{E}\bigl[f_N(R)\wt{h}^{\alpha_1,|\alpha|-\alpha_1}_{
n}(R;N)\bigr]
\nonumber
\\[-8pt]
\\[-8pt]
\nonumber
&=&\int_0^1 \phi_N(x)R^{}_{ n}(x)D_\alpha(\mathrm{d}x),
\end{eqnarray}
where
\[
\phi_N(x)=\sum_{r=0}^{N}\pmatrix{N\cr{r}}x^r(1-x)^{N-r}f_N(r)\geq0,
\]
and hence, by Gasper's theorem (Theorem~\ref{th:2}), $\rho_{ n}$ is
($\alpha_1,|\alpha|-\alpha_1$)-JPDS. Therefore, by Proposition \ref
{pdprop}, it is also $\alpha$-JPDS.
Finally, from the form of $\xi^{\alpha}_{ m}$, we know that
\[
{ r_{[ n]}}/{\ifact{(|\alpha|+ r)}{ n}}=\wh{\xi^{\alpha}_{ N}}(n)
\]
is $\alpha$-JPDS; thus (\ref{hpd}) is JPDS.
\end{pf}

\begin{remark}
\label{rm2} Notice that
\[
\frac{ r_{[ n]}}{\ifact{(|\alpha|+ r)}{ n}}
\]
is itself a positive-definite sequence for $D_{\alpha}.$ This is easy
to see directly from the representation~(\ref{inverse}) of
$\xi_{ m}^{\alpha}$ (we will consider more of it in Section {\ref
{sec:prob}}).

Since products of positive-definite sequences are positive
definite-sequences, then we have, as a completion to all previous results,

\begin{corollary}
{If $\{\rho_{ n}\}$ is positive-definite for $D_{\alpha}$, then (\ref
{hpd}) is positive-definite for both $D_{\alpha}$ and
$\operatorname{DM}_{\alpha}$.}
\end{corollary}
\end{remark}

\subsection{From Jacobi to Hahn positive-definite sequences via
discrete distributions}\label{sec:p2}
We have seen in Proposition~\ref{jhpd} that Jacobi positive-definite
sequences $\{\rho_{ n}\}$ can always be mapped to Hahn
positive-definite sequences of the form $\{\rho_{ n}\frac{N_{[
n]}}{(|\alpha|+N)_{( n)}}\}.$ We now show that a JPDS $\{\rho_{ n}\}$
is also HPDS when it is the image, via (\ref{gspd}), of a particular
class of discrete probability measures.
\begin{proposition}\label{prp:p2}
For every $N$ and $|\theta|>0$, let $\rho^{(N)}=\{\rho^{(N)}_{ n}\dvt
n\in
\mbb{Z}_{+}\}$ be of the same form (\ref{gspd}) for a probability mass
function $d^{(N)}=\{d_{ m}\dvt m\in\mbb{Z}_{+}\},$ such that $d_ l=0$ for
every $ l>N.$ Then $\rho^{(N)}$ is $\wt{\alpha}$-JPDS if and only if it
is $\wt{\alpha}$-HPDS for every $d$ and ${\alpha}\in\mbb{R}^d_+$, such
that $|{\alpha}|=|\theta|.$
\end{proposition}

\begin{pf}
By Lemma~\ref{l:hpd}, the sequence
\[
\biggl\{\frac{ m_{[ n]}}{\ifact{(|\alpha|+m)}{ n}}\biggr\}
\]
is HPDS (to convince oneself, take $\wt{\rho}$ as in Example \ref
{ex:gen} or in Example~\ref{ex:pf}, and take the limit as
$t\rightarrow
0$ or $z\rightarrow1$, resp.).

Now replace $ m$ with a random $M$ with distribution given by
$d^{(N)}.$ Then
\begin{eqnarray*}
0&\leq& \mbb{E}\Biggl[\sum_{ n=0}^{ m}\frac{ M_{[ n]}}{\ifact{(|\alpha
|+M)}{ n}}H_{ n}^{\alpha}\Biggr] \\
&=&\sum_{ n=0}^{N}\Biggl(\sum_{ m= n}^{N}d^{(N)}_{ m}\frac{ M_{[
n]}}{\ifact{(|\alpha|+M)}{ n}}\Biggr)H_{ n}^{\alpha},
\end{eqnarray*}
which proves the ``Hahn'' part of the claim. The ``Jacobi'' part is obviously proved by Proposition~\ref{gspd}.
\end{pf}

\section*{Acknowledgements}
Dario Span\`{o}'s research is partly supported by CRiSM, an
EPSRC/HEFCE-funded grant.
Part of the material included in this
paper (especially the first part) has been informally circulating for
quite a while, in form of notes, among other Authors. Some of them have
also used it for several interesting applications in statistics and
probability (see~\cite{DKSC08,KZ09}). Here we wish to thank
them for their helpful comments.

%
%

\printhistory

\end{document}
